\documentclass[a4paper, 11pt]{amsart}

\usepackage{amsmath,amssymb,amsthm,amscd,enumerate,setspace}
\usepackage[all]{xy}
\usepackage[dvips]{graphicx}

\usepackage{xcolor}
\usepackage{mdframed}
\newmdenv[backgroundcolor=yellow]{shaded}
\usepackage{hyperref}

\input xypic
\xyoption{all}
\newcommand{\rar}{\rightarrow}
\newcommand{\lar}{\longrightarrow}

\newcommand{\llar}{-\kern-5pt-\kern-5pt\longrightarrow}
\newcommand{\lllar}{-\kern-5pt-\kern-5pt\llar}

\newtheorem{Theorem}{Theorem}[section]
\newtheorem{Lemma}[Theorem]{Lemma}
\newtheorem{Corollary}[Theorem]{Corollary}
\newtheorem{Proposition}[Theorem]{Proposition}

\newtheorem{Conjecture}[Theorem]{Conjecture}

\newtheorem{Remark}[Theorem]{Remark}
\newtheorem{Example}[Theorem]{Example}
\newtheorem{Definition}[Theorem]{Definition}
\newtheorem{Question}[Theorem]{Question}
\newtheorem{Questions}[Theorem]{Questions}
\newtheorem{Exercise}[Theorem]{Exercise}

\newtheorem{Sticky Points}[Theorem]{Sticky Points}

\def\ann{\mbox{\rm ann}}
\def\Ass{\mbox{\rm Ass}}
\def\codim{\mbox{\rm codim}}
\def\coker{\mbox{\rm coker}}

\def\ds{\displaystyle}

\def\Ext{\mbox{\rm Ext}}

\def\gr{\mbox{\rm gr}}

\def\red{\mbox{\rm red}}

\def\grade{\mbox{\rm grade}}
\def\height{\mbox{\rm height }}
\def\Hom{\mbox{\rm Hom}}
\def\ker{\mbox{\rm ker}}

\def\e{\mathrm{e}}

\def\rme{\mathrm{e}}

\def\m{\mathfrak{m}}

\def\Spec{\mbox{\rm Spec}}

\def\Tor{\mbox{\rm Tor}}

\def\ZZ{\mathbb{Z}}

\def\AA{{\mathbf A}}

\def\BB{{\mathbf B}}
\def\CC{{\mathbf C}}

\def\RR{{\mathbf R}}
\def\SS{{\mathbf S}}
\def\FF{{\mathbf F}}

\def\LL{{\mathbf L}}

\def\TT{{\mathbf T}}
\def\ff{{\mathbf f}}
\def\pp{{\mathbf p}}

\def\g2{{\mathbf g}}

\def\aa{{\mathbf a}}

\def\zz{{\mathbf z}}

\def\H{{\mathrm H}}

\def\m{{\mathfrak m}}
\def\p{{\mathfrak p}}
\def\q{{\mathfrak q}}

\def\C{\mathcal{C}}
\def\B{\mathcal{B}}
\def\D{\mathcal{D}}

\def\Ddeg{\mbox{\rm ddeg}}

\def\ddeg{\mbox{\rm bideg}}

\def\tdeg{\mbox{\rm tdeg}}

\def\cdeg{\mbox{\rm cdeg}}
\def\codim{\mbox{\rm codim}}
\def\coker{\mbox{\rm coker}}

\def\tr{\mbox{\rm trace}}

\begin{document}

\title{\sc  Canonical Degrees of  Cohen-Macaulay Rings 
and Modules: a Survey }

\author{J. P. Brennan,
 L. Ghezzi,  J. Hong, L. Hutson and W. V. Vasconcelos}

\maketitle

\tableofcontents

\begin{abstract}
The aim  of this survey  is to discuss  invariants of Cohen-Macaulay local rings
that admit a canonical module.
 Attached
to each such ring $\RR$ with a canonical ideal $\C$, there are
  integers--the type of $\RR$, the reduction number of $\C$--that provide valuable metrics
  to express the deviation of $\RR$ from being a Gorenstein ring.
We  enlarge  this list with other integers--the roots of $\RR$ and several canonical degrees.  
 The
latter
 are multiplicity based functions of the Rees algebra of $\C  $. We give a uniform presentation of three degrees arising from common roots. Finally we experiment with ways to extend one of these degrees to rings where $\C$ is not necessarily an ideal.
\end{abstract}


\section{Introduction}

Let $(\RR, \m)$ be a Cohen-Macaulay local ring of dimension $d$ that has a canonical ideal $\C$. Our central viewpoint is to look
at the properties of $\C$ as a way to refine our understanding
of $\RR$.  
In \cite{blue1} 
several metrics are treated aimed at measuring the deviation from $\RR$ being Gorenstein, that is when
$\C \simeq \RR$. Here we explore another pathway but still with the same overall goal.
 Unlike \cite{blue1} 
 the approach here is
   arguably
 more suited for   computation in classes of algebras such as Rees algebras and monomial subrings.
 First however we outline the general
 underpin of these developments. 
 The organizing principle to set up a canonical degree
is to recast numerically   criteria for a Cohen-Macaulay ring  to be Gorenstein. 

\medskip

We shall now describe how this paper is organized. For a Cohen-Macaulay local ring $(\RR, \m)$ of dimension $d$ with a canonical ideal $\C$,
we are going to attach a non-negative integer $c(\RR)$ whose value reflects divisorial properties of $\C$
and provide for a stratification of the class of Cohen-Macaulay rings.
 We have noted two such functions in the current
literature (\cite{Herzog16}, \cite{blue1})  and here we will build a third degree. 

\medskip

In Section 2 we recall from the literature the needed blocks to put together  the degrees.  Section 3 quickly assembles three degrees and begins the
comparison of its properties. These assemblages turn out to provide effective symbolic calculation [we used Macaulay2 (\cite{Macaulay2}) in our
experiments] but turn out useful for theoretical calculations in special classes of rings. The new degree is labelled the {\em bi-canonical} degree of $\RR$ and 
is given by 
\[
\boxed{ \ddeg(\RR)= \deg(\C^{**}/\C) =
     \sum_{\tiny \height \p=1} \ddeg(\RR_{\p}) \deg(\RR/\p) = \sum_{\tiny \height \p=1} [\lambda(\RR_{\p}/\C_{\p}) - \lambda(\RR_{\p}/\C^{**}_{\p})] \deg(\RR/\p).}
  \] 
  This is a well-defined finite sum independent of the chosen canonical ideal $\C$. 
It leads immediately to comparisons to two other degrees, the {\em canonical} degree of \cite{blue1},
$\cdeg(\RR)= \deg(\C/(s))$ for a minimal reduction $(s)$ of $\C$ [in dimension one and suitably assembled as above to all dimensions],
 and the {\em residue} of $\RR$ of \cite{Herzog16}
$\tdeg(\RR) = \deg(\RR/\mbox{\rm trace}(\C))$, where $\mbox{\rm trace}(\C)$ is the trace ideal of $\C$.
Arising naturally is a
 comparison conjecture, that $\cdeg(\RR) \geq \ddeg(\RR)$.
We engage in a brief discussion on how to recognize that a codimension one ideal $I$ is actually a canonical ideal. We finally recall the notion of the {\em rootset}
of $\RR$ (\cite{blue1}), perhaps one of least understood sets attached to $\C$ and raise questions on how it affects the values of the degrees.

\medskip

We begin in Section 4 a study of algebras according to the values of one of the $c(\RR)$. If $c(\RR)=0$, for all canonical degrees, $\RR$ is Gorenstein in 
codimension one. It is natural to ask which rings correspond to small values of $c(\RR)$. In dimension one, $\cdeg(\RR)\geq r(\RR)-1$ and $\ddeg(\RR)\geq 1$,
where equality corresponding to the {\em almost Gorenstein } rings of  (\cite{BF97, GMP11, GTT15})
and 
{\em nearly Gorenstein} rings of \cite{Herzog16}, respectively.

\medskip

We begin in Sections 5, 6, 7, 8, 9, 10, 11, 12 calculations of $\cdeg(\RR)$ and
 $\ddeg(\RR)$ for various classes of algebras. Unlike the case of $\cdeg(\RR)$, already for hyperplane sections the 
behaviour of $\ddeg(\RR)$ is more challenging. Interestingly, for monomial rings $k[t^a, t^b, t^c]$ the technical difficulties are reversed. In two cases, augmented rings and 
[tensor] products, very explicit formulas are derived.  More challenging is the case of Rees algebras when we are often
 limited to deciding the vanishing of degrees. The most comprehensive results resolve around $\m:\m$. 

\medskip

The ring $\AA=\m:\m$ has a	 special role in the literature of low dimensional rings. For instance, if 
$\RR$ is a
Buchsbaum ring of dimension $\geq 2$ and
positive depth then $\AA$ is its $S_2$-ification (\cite[Theorem 4.2]{red}).

\medskip

Sections 13, 14, 15 discuss various possible generalizations and open questions.

\section{Setting up and calculating 
canonical degrees}\index{canonical degree; calculation}

\noindent

In this section we describe the  canonical degrees known to the authors and extend them to more general structures.

\subsubsection*{Divisorial basics of $\C$}
We are going to make use of the  basic facts expressed in the codimension one localizations of $\RR$.  

\begin{Theorem} \label{corecan} 
Let $(\RR, \m)$ be a local ring of dimension one
let $Q$ be its
total ring of fractions. Assume that $\RR$ has a canonical module $\C$.

\begin{enumerate}[{\rm (1)}]

\item $\RR$ has a canonical ideal if and only if the total ring of fractions of
 $\widehat{\RR}$ is Gorenstein.  \cite{Aoyama, BrodSharp, HK2}. 
 
 \smallskip

\item
 The $\m$-primary ideal $I$ is a canonical ideal if and only if $I:_{\RR} \m = I:_Q \m = (I, s)$ for some $s\in \RR$.  \cite[Theorem 3.3]{HK2}. 

\smallskip

\item $\RR$ is Gorenstein if and only if $\C$ is a reflexive module.
\cite[Corollary 7.29]{HK2}.

\smallskip

\item If $\RR$ is an integral domain with finite integral  closure then $I^{**}$ is integral over $I$.  \cite[Proposition 2.14]{CHKV}.

\end{enumerate}
\end{Theorem}

We often assume harmlessly 
 that $\RR$ has  an infinite 
 residue field and has arbitrary Krull dimension. For a finitely generated $\RR$-module $M$,
 the notation $\deg(M)=\rme_0(\m, M)$
  refers to the multiplicity defined by the $\m$-adic topology. The Cohen-Macaulay type of $\RR$ is denoted by $r(\RR)$. 

\begin{Remark}{\rm 
 Among the ways we can set up the comparison of $\C$ to  a principal ideal we have the following. 

 \begin{enumerate}[(1)]
 \item In dimension one,
 select an element $c$ of $\C$ and define $\cdeg(\RR) = \deg(\C/(c))$.  
  The choice should yield the same value for all 
 $\C$. In \cite{blue1} $(c)$ is chosen as a minimal reduction of $\C$ when then
 $\cdeg(\RR) = \rme_0(\C, \RR) - \deg(\RR/\C)$.
 
 \smallskip
 
 \item  The choice  is more straightforward  in one case: Set $\ddeg(\RR) = \deg(\C^{**}/\C)$, where $\C^{**}$ is the bidual of $\C$.
 
 \smallskip
 
 \item A standard metric is simply $\tdeg(\RR) =\deg(\RR/\tau(\C))$, where $\tau(\C)$ is the {\em trace} ideal of $\C$: $(f(x), f\in \C^{*}, x\in \C)$. It is often used
 to define the Gorenstein locus of $\RR$ (see \cite{Herzog16} for a discussion). For a method to calculate the
 trace of a module see \cite[Proposition 3.1]{Herzog16}, \cite[Remark 3.3]{Vas91}.

 \smallskip
 
 \item
  These are distinct  [in dimension $d>1$]
   numbers, which are independent of the choice of $\C$, that share a common property:
  
\smallskip
  
  \begin{enumerate}[(i)]
\item  $\RR$ is Gorenstein if and only if one of $\cdeg(\RR)$, $\ddeg(\RR)$ or $\tdeg(\RR)$ vanishes (in which case all three vanish).
  
  \smallskip
  
 \item 
  If $\RR$ is not Gorenstein, $\cdeg(\RR) \geq r(\RR) - 1\geq 1 $,  $\tdeg(\RR) \geq 1$, $\ddeg(\RR) \geq 1$, and the minimal values are attained. 
   In dimension one  $\ddeg = \tdeg$, see
   Proposition~\ref{resisddeg1}.
 \end{enumerate}
 
 \smallskip
 
 \item The cases when the minimal values are reached have the following designations:
 
 \smallskip
 
 \begin{enumerate}[(i)]
 \item $\cdeg(\RR)$ was introduced in \cite{blue1} and called the canonical degree of $\RR$:
    $\cdeg(\RR) = r(\RR) - 1$ if and only if $\RR$ is  an {\em almost Gorenstein ring} 
    (\cite{GMP11, GTT15}).   
 
  \smallskip
  
  \item $\tdeg(\RR)$ was introduced in \cite{Herzog16} and called it the {\em residue} of $\RR$: $\mbox{\rm res}(\RR) $.
  It can also be called the trace degree
   of $\RR$.
  Another possible terminology is to call it the {\em anti-canonical} degree of $\RR$: $\tdeg(\RR) = 1$ 
  if and only if $\RR$ is a {\em nearly Gorenstein ring}.

  \smallskip
  
  \item $\ddeg(\RR)$ was introduced in \cite{bideg}
  and  called the  bi-canonical  
  degree of $\RR$: $\ddeg(\RR) = 1$ if and only
  if $\RR$ is a
    so-called
     {\em Goto ring}.
      
 \end{enumerate}
\end{enumerate}
}\end{Remark}

\medskip

 There are some  relationships among these invariants.

\begin{Proposition}{\rm [J. Herzog, personal communication]}\label{resisddeg1} Let $\RR$ be Cohen-Macaulay local ring of dimension $1$ with a canonical ideal $\C$. Then
\[
\boxed{ \ddeg(\RR) = 
\lambda(\RR/\mbox{\rm trace}(\C)).}
\]
\end{Proposition}

\begin{proof} We will show that
\[ \lambda(\C^{**}/\C) =
\lambda(\RR/\mbox{\rm trace}(\C)).\]
From $\tr(\C) = \C\cdot
 \C^{*}$, we
 have 
\[  \Hom(\tr(\C),\C)=  \Hom(\C \cdot \C^{*},\C)= \Hom(\C^{*}, \Hom(\C,\C))  = \C^{**}.\]
 Now dualize the
  exact sequence
\[
0 \rar
\tr(\C)\lar  \RR \lar \RR/\tr(\C)\rar 0,\]   into $\C$
to obtain, 
\[0=\Hom(\RR/\tr(\C),\C)\rar  \C=\Hom(\RR,\C)\rar \C^{**}=\Hom(\tr(\C),\C)\rar  \Ext^1(\RR/\tr(\C),\C)\rar  0,\]
which 
 shows that $\C^{**}/\C$ and $\Ext^1(\RR/\tr(\C),\C)$ are isomorphic. Since by local duality $\Ext^1(\cdot,\C)$ is self-dualizing on modules of
 finite support, 
$\ddeg(\RR)= \lambda(\RR/\tr(\C)).$
\end{proof}

A similar
 argument applies to a slightly different class of ideals. We say that an ideal $I$ is {\em closed} if
$\Hom(I,I)=\RR$ in the terminology of \cite{BV1}.

\begin{Proposition} \label{resddeg3}
 Let $\RR$ be a local ring of dimension one and finite integral closure,
and let $I$ be a closed ideal. If
$I$ is reflexive then $I$ is principal.
\end{Proposition}

\begin{proof}
We argue by contradiction.
From $\tr(I) = L=I\cdot
 I^{*}$, we
 have 
\[  \Hom(\tr(I),\RR)=  \Hom(I \cdot I^{*},\RR)= \Hom(I\otimes I^{*},\RR)= \Hom(I, \Hom(I^{*},\RR) ) =
 \Hom(I,  I^{**})  = \RR,\] 
 and thus $L^{**} = \RR x$. If $x$ is a unit $L^{*}= \RR$ and therefore $L$
  has height greater than one,
 which is not possible. Thus we have
 $ L = Mx$, where $M \subset \m$.
 Since $L^{**}=\RR x$ is integral over $L=Mx$  (\cite[Proposition 2.14]{CHKV}),
 we have for some positive integer $n$
 \[ \RR x^n = Mx\RR x^{n-1},\]
 and thus $\RR = M$, which is impossible.
 \end{proof}

\section{The basic canonical degree}

\noindent  Let $(\RR,\m)$ be a Cohen-Macaulay local ring. Suppose that $\RR$  has a canonical ideal $\C$.  In this setting we
 introduce a numerical degree for $\RR$ and study its properties. The starting point of our discussion is the
following elementary observation. We denote the length function
by $\lambda$.

 \begin{Proposition}\label{cdeg}
 Let $(\RR,\m)$ be a $1$-dimensional Cohen-Macaulay local ring with  a canonical ideal $\C$.
 Then the integer $\cdeg(\RR)=\e_0(\C) -\lambda(\RR/\C)$ is independent of the canonical ideal $\C$.
 \end{Proposition}

\begin{proof}  If $x$ is an indeterminate over $\RR$, in calculating these differences we may pass from $\RR$
to $\RR(x) = \RR[x]_{\m \RR[x]}$, in particular we may assume that the ring has an infinite residue field.
Let $\C$ and $\mathcal{D}$ be two canonical ideals. Suppose $(a)$ is a minimal reduction of $\C$.
Since $\mathcal{D} \simeq \C$ (\cite[Theorem 3.3.4]{BH}), $\mathcal{D} = q \C$ for some fraction $q$.  If $\C^{n+1} = (a) \C^n$
 by multiplying it by $q^{n+1}$,  we get $\mathcal{D}^{n+1} =(qa) \mathcal{D}^n$, where $(qa) \subset \mathcal{D}$. Thus $(qa)$ is a reduction of $\mathcal{D}$ and
$\C/(a) \simeq \mathcal{D}/(qa)$. Taking their co-lengths we have
\[ \lambda(\RR/(a)) - \lambda(\RR/ \C) = \lambda(\RR/(qa)) - \lambda(\RR/ \mathcal{D}).\]
  Since $\lambda(\RR/(a)) = \e_0(\C)$ and $\lambda(\RR/(qa)) = \e_0(\mathcal{D})
 $,
we have
\[ \e_0(\C) - \lambda(\RR/ \C) = \e_0( \mathcal{D}) - \lambda(\RR/ \mathcal{D}). \qedhere \]
  \end{proof}

 We can define $\cdeg(\RR)$ in full generality as follows.

 \begin{Theorem}\label{gencdeg1} Let $(\RR, \m)$ be a Cohen-Macaulay local ring  of dimension $d \geq 1$  that has a canonical ideal $\C$.
  Then
  \[  \boxed{ \cdeg(\RR)= \sum_{\tiny \height \p=1} \cdeg(\RR_{\p}) \deg(\RR/\p) = \sum_{\tiny \height \p=1} [\e_{0}(\C_{\p}) - \lambda((\RR/\C)_{\p})] \deg(\RR/\p)}
  \] is a well-defined finite sum independent of the chosen canonical ideal $\C$. In particular, if $\C$ is equimultiple with a minimal reduction $(a)$, then
  \[  \cdeg(\RR) = \deg(\C/(a)) = \e_0(\m, \C/(a)).\]
 \end{Theorem}

\begin{proof} By Proposition \ref{cdeg}, the integer $\cdeg(\RR_{\p})$ does not depend on the choice of a canonical ideal of $\RR$. Also $\cdeg(\RR)$ is a finite sum since, if $\p \notin \mbox{\rm Min}(\C)$, then $\C_{\p}=\RR_{\p}$ so that $\RR_{\p}$ is Gorenstein. Thus $\cdeg(\RR_{\p})=0$. The last assertion follows from the associativity formula:
\[ \cdeg(\RR) =  \sum_{\tiny \height \p=1} \lambda( (\C/(a))_{\p}) \deg(\RR/\p) =  \deg(\C/(a)). \qedhere \]
\end{proof}

\begin{Definition}\label{defcdeg}{\rm
Let $(\RR, \m)$ be a Cohen-Macaulay local ring  of dimension $d \geq 1$ that has a canonical ideal.
  Then the {\em canonical degree of $\RR$} is the integer
  \[ \cdeg(\RR)= \sum_{\tiny \height \p=1} \cdeg(\RR_{\p}) \deg(\RR/\p). \index{canonical degree}
  \]
  }\end{Definition}

\begin{Corollary}
$\cdeg(\RR)\geq 0$ and vanishes if and only if  $\RR$ is Gorenstein in codimension $1$.
\end{Corollary}

\begin{Corollary}\label{cdegr1}
Suppose that the canonical ideal of $\RR$ is equimultiple. Then we have the following.
\begin{enumerate}[{\rm (1)}]
\item $\cdeg(\RR) \geq r(\RR)-1$.

\smallskip

\item $\cdeg(\RR) =0$ if and only if $\RR$ is Gorenstein.
\end{enumerate}
\end{Corollary}

\begin{proof}
Let $(a)$ be a minimal reduction of the canonical ideal $\C$. Then
\[ \cdeg(\RR) = \e_0(\m, \C/(a))\geq \nu(\C/(a))=r(\RR)-1.\]
If $\cdeg(\RR)=0$ then $r(\RR)=1$, which proves that $\RR$ is Gorenstein.
\end{proof}

 Now we extend the above result to a more general class of ideals when the ring has dimension one. We recall that if $\RR$ is a $1$-dimensional 
 Cohen-Macaulay local ring, $I$ is an $\m$-primary ideal with  minimal reduction $(a)$, 
 then the reduction number of $I$ relative to $(a)$ is independent of the reduction 
 (\cite[Theorem 1.2]{Trung87}). It will be denoted simply by $\red(I)$.

\begin{Proposition} Let $(\RR, \m)$ be a $1$-dimensional  Cohen-Macaulay local ring which is not a valuation ring. Let  $I$ be an irreducible $\m$-primary ideal such that $I \subset \m^2$.  If $(a)$
is a minimal reduction of $I$, then $\lambda(I/(a)) \geq r(\RR)-1$. In the case of equality, $\red(I) \leq 2$.
\end{Proposition}

\begin{proof}
Let $L = I:\m$ and $N=(a): \m$. Then ${\ds \lambda(L/ I) = r(\RR/ I) = 1}$ and ${\ds \lambda( N / (a) ) = r(\RR)}$. Thus, we have
\[ r(\RR) \leq  \lambda(L/ N )  +  \lambda( N / (a) ) = \lambda(L/(a) ) = \lambda(L/ I)  + \lambda( I/(a) ) = 1 + \lambda (I/(a) ),\]
which proves that $\lambda(I/(a)) \geq r(\RR)-1$.

\medskip

\noindent Suppose that $\lambda(I/(a)) = r(\RR)-1$. Then $L=N$. By \cite[Lemma 3.6]{CHV}, $L$ is integral over $I$. Thus, $L$ is integral over $(a)$. By \cite[Theorem 2.3]{CP95}, $\red(N)=1$. Hence $L^2 = a L$.  Since $\lambda(L/ I) = 1 $, by
\cite[Proposition 2.6]{GNO}, we have $I^3 = a I^2$.
\end{proof}

\section{Extremal values of the canonical degree}
\noindent
We examine in this section extremal values of the canonical degree.
First we recall the definition of almost Gorenstein rings  (\cite{BF97, GMP11, GTT15}).

\begin{Definition}{\rm (\cite[Definition 3.3]{GTT15})  A Cohen-Macaulay local ring $\RR$ with a canonical module $\omega$  is said to be an {\em almost Gorenstein} ring if there exists an exact sequence of $\RR$-modules
${\ds 0\rightarrow\RR\rightarrow \omega \rightarrow X\rightarrow  0}$
such that $\nu(X)=\e_0(X)$.
}\end{Definition}

\begin{Proposition}\label{almostg}
Let $(\RR,\m)$ be a Cohen-Macaulay local ring with a canonical ideal $\C$. Assume that $\C$ is equimultiple.
If $\cdeg (\RR) = r(\RR) -1$, then $\RR$ is an almost Gorenstein ring. In particular,  if $\cdeg (\RR) \le 1$, then $\RR$ is an almost Gorenstein ring.
\end{Proposition}

\begin{proof}
We may assume that $\RR$ is not a Gorenstein ring. Let $(a)$ be a minimal reduction of $\C$.
Consider the exact sequence of $\RR$-modules
\[ 0 \to \RR \overset{\varphi}{\to} \C \to X \to 0, \;\; \mbox{where} \;\; \varphi (1) = a. \] Then $\nu(X)=r(\RR) -1=\cdeg (\RR)=\e_0(X)$. Thus, $\RR$ is an almost Gorenstein ring.
\end{proof}

\begin{Proposition}\label{1dimalmostg}
Let $(\RR,\m)$ be a $1$-dimensional Cohen-Macaulay local ring with a canonical ideal $\C$. Then $\cdeg(\RR)= r(\RR)-1$ if and only if $\RR$ is an almost Gorenstein ring.
\end{Proposition}

\begin{proof}
It is enough to prove that the converse holds true. We may assume that $\RR$ is not a Gorenstein ring.  Let $(a)$ be a minimal reduction of $\C$.
Since $\RR$ is almost Gorenstein, there exists an exact sequence of $\RR$-modules
\[ 0 \to \RR \overset{\psi}{\to} \C \to Y \to 0 \;\; \mbox{such that} \;\; \nu(Y)=\e_0(Y).\]  Since $\dim(Y)=0$ by
\cite[Lemma 3.1]{GTT15}, we have that $\m Y=(0)$.
Let $b = \psi (1) \in \C$ and set $\q = (b)$. Then $\m \q \subseteq \m \C \subseteq \q$. Therefore, since $\RR$ is not a DVR and $\l_\RR(\q/\m \q) = 1$, we get $\m \C = \m  \q$, whence the ideal $\q $ is a reduction of $\C$ by the Cayley-Hamilton theorem, so that $\cdeg (\RR) = \e_0(Y) = \nu(Y) =r(\RR) - 1$.
\end{proof}

Now we consider the general case when a canonical ideal is not necessarily equimultiple.

\begin{Lemma}\label{prop1}
Let $(\RR, \m)$ be a Cohen-Macaulay local ring of dimension $d \geq 1$ with infinite residue field and a canonical ideal $\C$.
Let $a$ be an element of  $\C$ such that {\rm (i)} for $\forall \p \in \Ass (\RR/\C)$ the element $\frac{a}{1}$ generates a reduction of $\C R_\p$, {\rm (ii)} $a$ is $\RR$-regular, and {\rm (iii)} $a \not\in \m \C$.  Let ${\ds Z=\{ \p \in \Ass(\C/(a)) \mid \C \not\subseteq \p \} }$.

\begin{enumerate}[{\rm (1)}]
\item ${\ds \cdeg (\RR) = \deg(\C/(a)) - \sum_{\p \in Z}\lambda( (\C/(a))_\p) \deg(\RR/\p)}$.

\smallskip

\item $\cdeg (\RR) = \deg(\C/(a))$ if and only if $\Ass(\C/(a)) \subseteq V(\C)$.
\end{enumerate}
\end{Lemma}

\begin{proof}
It follows from Theorem \ref{gencdeg1} and  ${\ds \deg(\C/(a)) = \sum_{\tiny \height \p=1} \lambda((\C/(a))_\p)\deg(\RR/\p)}$.
\end{proof}

\begin{Theorem}\label{almostg2} With the same notation given in Lemma {\rm \ref{prop1}}, suppose that $\Ass(\C/(a)) \subseteq V(\C)$. Then  $\cdeg (\RR) =r(\RR) - 1$ if and only if  $\RR$ is an almost Gorenstein ring.
\end{Theorem}

\begin{proof}
It is enough to prove the converse holds true. We may assume that $\RR$ is not a Gorenstein ring. Choose an exact sequence
\[ 0 \to \RR \to \C \to Y \to 0 \]
such that $\deg(Y)= \nu(Y) = r(\RR) - 1$. Since $\Ass(\C/(a)) \subseteq V(\C)$, we have $\deg(Y) \geq \deg(\C/(a))$.  By Theorem \ref{cdegr1} and Lemma \ref{prop1}, we obtain the following.
\[ r(\RR) -1 \leq \cdeg(\RR) = \deg(C/(a)) \leq \deg(Y) = r(\RR)-1. \qedhere \]
\end{proof}

\begin{Remark}{\rm If $\RR$ is a non-Gorenstein normal domain, then its canonical ideal cannot be equimultiple.}
\end{Remark}

\begin{proof}
 Suppose that a canonical ideal $\C$ of $\RR$ is equimultiple, i.e., $\C^{n+1}= a \C^{n}$.  Then we would have an equation $(n+1) [ \C] = n[\C]$ in its divisor class group. This means that $[\C]=[0]$. Thus, $\C \simeq \RR$. Hence $\C$ cannot be equimultiple.
\end{proof}

\section{The bi-canonical degree}

 The approach in \cite{blue1} is dependent on finding minimal reductions.
 We pick here one that seems  
particularly amenable 
to computation. Let $\C^{*} = \Hom(\C, \RR)$ be the dual of $\C$ and $\C^{**}$ its bidual.  [In general, in writing $\Hom_{\RR}$ we omit the symbol for the
 underlying ring.]
In the natural embedding
\[ 0 \rar \C \lar \C^{**} \lar \B \rar 0, \]
$\B$ remains unchanged when $\C$ is replaced by another canonical module, say $D = s\C$ for a regular element $s\in Q$. 
$\B$ vanishes if and only if $\RR$ is Gorenstein as indicated above. It is easy to see that a similar observation can be
made if $d>1$. That is, $\B = 0$ if and only if $\RR$ is Gorenstein in codimension $1$.  $\B$ embeds into the Cohen-Macaulay module $\RR/\C$ that  has  dimension $d-1$, and thus $\B$ either is zero or its associated primes are associated primes of $\C$, all of which have codimension one.

\medskip

\noindent Like in \cite{blue1}, we would like to explore the length of $\B$, $\ddeg(\RR) =\lambda(\B)$ which we view as a degree,
 in dimension $1$ and $\deg(\B)$ in general. We stick to $d=1$ for the time being. We would like some interesting examples and examine relationships to the other metrics of $\RR$. We do not have best  name for this degree, but we could also
 denote it by $\Ddeg(\RR)$
  (at least the double `d' as a reminder  of `double dual').

 \medskip
 
Let us formalize these observations as the following.

 \begin{Theorem}\label{genddeg1} Let $(\RR, \m)$ be a Cohen-Macaulay local ring  of dimension $d \geq 1$  that has a canonical ideal $\C$.
  Then
  \[ \boxed{ \ddeg(\RR)= \deg(\C^{**}/\C) =
     \sum_{\tiny \height \p=1} \ddeg(\RR_{\p}) \deg(\RR/\p) = \sum_{\tiny \height \p=1} [\lambda(\RR_{\p}/\C_{\p}) - \lambda(\RR_{\p}/\C^{**}_{\p})] \deg(\RR/\p)}
  \] is a well-defined finite sum independent of the chosen canonical ideal $\C$. 
  Furthermore,
  $\ddeg(\RR)\geq 0$ and vanishes if and only if  $\RR$ is Gorenstein in codimension $1$.
\end{Theorem}

\subsubsection*{Comparison of canonical degrees} \label{compofdegs}
If $(c)$ 
is a minimal reduction of $\C$ how to compare
\[ \cdeg(\RR) = \lambda (\RR/(c)) - \lambda(\RR/\C) \Leftrightarrow \lambda(\RR/\C) -\lambda(\RR/\C^{**}) = \ddeg(\RR)\]

The point to be raised is: which is more approachable,  $\e_0(\C)$ or $\lambda(\RR/\C^{**})$? we will argue,  according to the method 
of computation below, that the latter is more efficient which would be
demonstrated if the following conjecture were settled.

\begin{Conjecture}{\rm [Comparison Conjecture]}
\label{cdegvsfddeg}
{\rm If $\dim \RR =1$
the following inequality holds
\[ \boxed{
  \cdeg(\RR) \geq \ddeg(\RR).}\] That is from the diagram
  \[
\diagram
(c) \rline & \C  \rline & 
\C^{**} 
\enddiagram
\]
where $\C^{**} = (c): ((c):\C)$, we have $\lambda(\C/(c)) \geq \lambda(\C^{**}/\C)$.
Alternatively  
\[ \e_0(\C) + 
\lambda(\RR/\C^{**})  \geq 2 \cdot \lambda(\RR/\C).\] 
}\end{Conjecture}

\noindent This would imply, by the associativity formula, that
the inequality holds in all dimensions.

\medskip

\subsubsection*{Computation of duals  and biduals} Let $I$ be a regular ideal of the Noetherian ring $\RR$. If $ Q$ is the total ring of fractions of $\RR$ then
\[ \Hom(I, \RR) = \RR:_Q I.\]
A difficulty is that computer systems such as Macaulay2  (\cite{Macaulay2})
are set to calculate quotients of the form $A:_\RR B$ for two ideals
 $A,B\subset \RR$, which is done with calculations of syzygies. This applies especially  in the case of the ring $\RR = k[x_1, \ldots, x_d]/P$ where $k$ is an appropriate field.
  To benefit of the efficient  {\em quotient} command of this system we formulate the problem as follows.

\begin{Proposition} \label{bidual}
Let  $I \subset \RR$ and suppose  $a$ is a regular element of $I$. Then
    
\begin{enumerate}[{\rm (1)}]
\item $I^{*} = \Hom(I,\RR) = a^{-1} ((a) :_\RR I)$.

\smallskip

\item
     $I^{**} = \Hom(\Hom(I, \RR), \RR) =
          (a):_\RR ((a):_\RR I)$ = annihilator of $\Ext^1_{\RR}(\RR/I, \RR)$.
          
\smallskip

\item $\tau(I) = a^{-1} I \cdot (a):_\RR I$.     
\end{enumerate}
\end{Proposition}

\begin{proof}
(1) If $q \in \Hom(I, \RR) = \RR:_Q I$, then
\[ qI \subset \RR \simeq qaI \subset (a) \simeq qa \in (a):_\RR I \simeq q \in a^{-1} ((
a):_\RR I) 
\] (3) Follows from the calculation
\[q \in I^{**} = [a^{-1} ((a):_\RR I)]^{*}= a[\RR:_Q ((a):_\RR I)] = a[a^{-1} (
(a):_\RR ((a):_\RR I)] = (a):_\RR ((a):_\RR I). \]
See also \cite[Remark 3.3]{Vas91}, \cite[Proposition 3.1]{Herzog16}.
\end{proof}

\subsubsection*{Computation of $\Hom(A, B)=A:_QB$.} Let $A$ and $B$ be ideals of $\RR$. A question is 
how to trick the ordinary quotient command to do these computations. There are important cases needed, such as $\Hom(A,A)$. If $B = (b_1, \ldots, b_n)$, $b_i$ regular element of $\RR$, then 
\[A:_Q B = \bigcap_{i=1}^n  A:_Q b_i=
\bigcap_{i=1}^n  Ab_i^{-1},\]
  so setting $b = \prod_{i=1}^n b_i$ and $\widehat{b}_i = b b_i^{-1}$, we have
  \[  A:_QB = b^{-1} \bigcap_{i=1}^n A \widehat{b_i}.\]
A case of interest is when $I = \m$: $\AA = \Hom(\m,\m)$. If $\RR$ is a Cohen-Macaulay local ring of dimension one that is
not a DVR, then $\AA = \RR:\m$:  $\m\cdot \Hom(\m,\RR) \subset \m$ and therefore
\[ \AA = \m:_Q \m = \RR:_{Q} \m= 1/x \cdot ((x):_{\RR} \m),\]
where $x$ is a regular element of $\m$.

\begin{Example}\label{C1C2}
{\rm 
Here is an example from \cite{blue1}. Let $L=(X^2-YZ, Y^2-XZ, Z^2-XY)$ and $\RR=A/L$. Let $x, y, z$ be the images of $X, Y, Z$ in $\RR$. Then $\C=(x, z)$ is a canonical ideal of 
\[ \C^3 = ( x^3, x^{2}z, xz^{2}, z^{3} ) = (x^3, x^{2}z, xz^{2} ) = x \C^2,\]
which proves that $\red(\C)=2$. Note that
 \[ \e_0(\C) = \lambda ( \RR/x \RR)  = \lambda(A/(x+L)) = \lambda( A/ (x, y^{2}, yz, z^{2})) = 3. \]
}\end{Example}

\subsubsection*{Recognition of canonical ideals} These methods permit answering the following question. 
  Given an ideal $\C$,  is it a canonical ideal? These  observations are influenced by the discussion in \cite[Section 2]{Elias}. For other methods, see \cite{GMP11}.

\begin{Proposition}  \label{recog}
Let $(\RR,\m)$ be a one-dimensional Cohen-Macaulay local ring and let $Q$ be its total ring of fractions. Then an 
$\m$-primary ideal $I$ is a canonical ideal if $I$ is  an irreducible ideal  and $\Hom(I,I) = \RR$.
 \end{Proposition}
 
\begin{proof} Note first that
 if $q \in I:_Q \m$ then $q I \subset I$, that 
is $q\in \Hom(I, I)$. Now we invoke Theorem~\ref{corecan}(2).
  \end{proof}  
  
 \begin{Corollary} \label{recogcor} Let $(\RR, \m)$ be a Cohen-Macaulay local ring of dimension one.
 \begin{enumerate}[{\rm (1)}]
  
\item The $\m$-primary ideal $I$ is a canonical ideal if and only if both $I$ and $xI$ are irreducible for any regular element $x\in \m$. 
 $($This is
{\cite[Proposition  2.2]{Elias}}.$)$

\smallskip

\item  $\m$ is a canonical ideal if and  only if $\RR$ is a discrete valuation ring.
 \end{enumerate}
\end{Corollary}

\begin{proof}
If $q\in I:_Q \m$, $qx\subset I \subset x\RR$, so $q\in \RR$. Thus $q\in I:_{\RR} \m$. Since 
$I$ is irreducible
$I:_{\RR} \m = (I,s)$ and we can invoke 
again
Theorem~\ref{corecan}(2). Apply the previous assertion to $xI$.  The converse is well-known.

\medskip 

\noindent Let $x$ be a regular element in $\m$. Since $\dim \RR = 1$ there is a positive integer $n$ such that $\m^n \subset x\RR$ but 
$\m^{n-1} \not \subset x\RR$. If $n = 1$, $\m = x\RR$ there is nothing else to prove. If $n> 1$, we have that
$L=\m^{n-1} x^{-1} $ satisfies $L \cdot \m $ is an ideal of $\RR$ so that either $L \cdot \m = \RR$ and hence $\m$ is invertible, or
$L \cdot \m \subset \m$ that means $L \subset \Hom(\m, \m) = \RR$. Thus $\m^{n-1} x^{-1} \subset \RR$ so
$\m^{n-1} \subset x\RR$, which is a contradiction.  
\end{proof}

\section{Canonical index}
\noindent
Throughout the section, let $(\RR, \m)$ be a Cohen-Macaulay local ring of dimension $d \geq 1$ with infinite residue field and suppose that a canonical ideal $\C$ exists.
We begin by showing  that the reduction number of a canonical ideal of $\RR$ is an invariant of  the ring.

\begin{Proposition}\label{Ca1}
Let $\C$ and $\mathcal{D}$ be canonical ideals of $\RR$. Then $\red(\C)=\red(\mathcal{D})$.
\end{Proposition}

\begin{proof}
 Let $K$ be the total ring of quotients of $\RR$. Then there exists $q \in K$ such that $\mathcal{D}=q\C$. Let $r = \red(\C)$ and $J$ a minimal reduction of $\C$ with $\C^{r+1}=J \C^{r}$. Then
\[ \mathcal{D}^{r+1}= (q \C)^{r+1}= q^{r+1} (J \C^{r}) = (qJ)(q \C)^{r}= qJ \mathcal{D}^{r}\] so that $\red(\mathcal{D}) \leq \red(\C)$. Similarly, $\red(\C) \leq \red(\mathcal{D})$.
\end{proof}

\begin{Definition}{\rm
Let $(\RR, \m)$ be a Cohen-Macaulay local ring of dimension $d \geq 1$ with a canonical ideal $\C$. The {\em canonical index} of $\RR$ is the reduction number of the canonical ideal $\C$ of $\RR$ and is denoted by $\rho(\RR)$.
}\end{Definition}

\begin{Remark}\label{R}{\rm Suppose that $\RR$ is not Gorenstein.
The following are known facts.
\begin{enumerate}[(1)]
\item If the canonical ideal of $\RR$ is equimultiple, then $\rho(\RR) \neq 1$.

\smallskip

\item If $\dim \RR=1$ and $\e_{0}(\m) =3$, then $\rho(\RR)=2$.

\smallskip

\item If $\dim \RR=1$ and $\cdeg(\RR)=r(\RR) -1$, then $\rho(\RR)=2$.
  \end{enumerate}
}\end{Remark}

\begin{proof}
(1) Suppose that $\rho(\RR)=1$. Let $\C$ be a canonical ideal of $\RR$ with $\C^2 = a\C$. Then  $\C a^{-1} \subset \Hom(\C , \C) = \RR$ so that $\C = (a)$. This is a contradiction.

\medskip

\noindent (2) It follows from the fact that, if   $(\RR, \m)$ is a $1$--dimensional Cohen-Macaulay ring and $I$ an $\m$--primary ideal, then $\red(I) \leq \e_{0}(\m)-1$.

\medskip

\noindent (3) It follows from Proposition~\ref{1dimalmostg} and \cite[Theorem 3.16]{GMP11}.
\end{proof}

\subsubsection*{Sally module}   We examine briefly the Sally module associated to
the canonical ideal $\C$ in rings of dimension $1$. Let $Q=(a)$
be a minimal reduction of  $\C$ and consider the exact sequence of finitely generated $\RR[Q\TT]$-modules
\[ 0 \rar \C \RR[Q\TT] \lar \C \RR[\C\TT] \lar S_Q(\C) \rar 0. \]
Then the Sally module  $S=S_Q(\C) = \bigoplus_{n\geq 1} \C^{n+1}/\C Q^{n}$ of $\C$ relative to $Q$ is Cohen-Macaulay and, by
 \cite[Theorem 2.1]{red}, we have
 \[\e_1(\C) = \cdeg(\RR) + \sum_{j=1}^{\rho(\RR)-1} \lambda(\C^{j+1}/a\C^j) = \sum_{j=0}^{\rho(\RR) -1} \lambda(\C^{j+1}/a\C^j). \]

\begin{Remark} {\rm Let $\RR$ be a $1$-dimensional Cohen-Macaulay local ring with a canonical ideal $\C$. Then the multiplicity of the Sally module $s_0(S)=\e_1(\C)-\e_0(\C)+\lambda(\RR/\C)=\e_1(\C)-\cdeg(\RR)$ is an invariant of the ring $\RR$, by \cite[Corollary 2.8]{GMP11} and Proposition \ref{cdeg}.}
\end{Remark}

The following property of Cohen-Macaulay rings of type $2$ is a useful calculation that we will use to characterize rings with minimal canonical index.

\begin{Proposition}\label{C2}
Let $\RR$ be a $1$-dimensional Cohen-Macaulay local ring with a canonical ideal $\C$.  Let $(a)$ be a minimal reduction of $\C$. If $\nu(\C)=2$, then $\lambda(\C^{2}/a \C)=\lambda(\C/(a))$.
\end{Proposition}

\begin{proof}
Let $\C=(a, b)$ and consider the exact sequence
\[ 0 \rar Z \rar \RR^{2} \rar \C \rar 0,\]
where $Z=\{ (r,s) \in \RR^{2} \mid ra+sb =0 \}$.  By tensoring this exact sequence with $\RR/ \C$, we obtain
\[ Z/\C Z  \stackrel{g}{\rar} (\RR/\C)^{2} \stackrel{h}{\rar} \C/\C^{2} \rar 0.\]
Then we have
\[ \ker(h)= \mbox{Im}(g) \simeq (Z/ \C Z)/ ( (Z \cap \C \RR^{2})/ \C Z ) \simeq Z/(Z \cap \C \RR^{2}) \simeq  (Z/B)/ ( (Z \cap \C \RR^{2})/B ),\]
where $B= \{ (-bx, ax) \mid x \in \RR \}$.

\medskip

\noindent We claim that $Z \cap \C \RR^{2} \subset B$, i.e., $\delta(\C)= (Z \cap \C \RR^{2})/B =0$. Let $(r, s) \in Z \cap \C \RR^{2}$. Then
\[ ra + sb =0 \;\; \Rightarrow \;\;  \frac{s}{a} \cdot b = -r \in \C \; \mbox{and} \; \frac{s}{a} \cdot a = s \in \C.\]
Denote the total ring of fractions of $\RR$ by $K$. Since $\C$ is a canonical ideal, we have
\[ \frac{s}{a} \in \C :_{K} \C =\RR.\]
Therefore
\[ (r, s) = \left( -b \cdot \frac{s}{a}, \; a \cdot \frac{s}{a} \right) \in B.\]
Hence $\ker(h)  \simeq Z/B =\H_{1}(\C)$ and we obtain the following exact sequence
\[ 0 \rar \H_{1}(\C) \rar (\RR/ \C)^{2} \rar \C/ \C^{2} \rar 0.\]
Next we claim that $\lambda(\H_{1}(\C))= \lambda(\RR/ \C)$.  Note that $\H_{1}(\C) \simeq ((a):b)/(a)$ by  mapping $(r,s)+B$ with $ra+sb=0$ to $s+(a)$. Using the exact sequence
\[ 0 \rar ((a):b)/(a) \rar \RR/(a) \stackrel{\cdot b}{\rar}  \RR/(a) \rar  \RR/\C \rar 0,\]
 we get
 \[ \lambda(\RR/\C) =  \lambda( ((a):b)/(a) ) = \lambda(\H_{1}(\C)).\]
Now, using the exact sequence
\[ 0 \rar \H_{1}(\C) \rar (\RR/\C)^{2} \rar \C/\C^{2} \rar 0, \]
we get
\[ \lambda(\C/\C^{2}) = 2 \lambda(\RR/\C) - \lambda(\H_{1}(\C)) = \lambda(\RR/\C).\]
Hence,
\[ \lambda(\C^{2}/a \C) = \lambda(\C/a \C) - \lambda(\C/ \C^{2}) = \lambda(\C/a \C) - \lambda	(\RR/ \C) = \lambda(\C/a \C) - \lambda	((a)/a \C) = \lambda(\C/(a)).\]
Therefore
\[ (r, s) = \left( -b \cdot \frac{s}{a}, \; a \cdot \frac{s}{a} \right) \in B.\]
Hence $\ker(h)  \simeq Z/B =\H_{1}(\C)$ and we obtain the following exact sequence
\[ 0 \rar \H_{1}(\C) \rar (\RR/ \C)^{2} \rar \C/ \C^{2} \rar 0.\]
Next we claim that $\lambda(\H_{1}(\C))= \lambda	(\RR/ \C)$.  Note that $\H_{1}(\C) \simeq ((a):b)/(a)$ by  mapping $(r,s)+B$ with $ra+sb=0$ to $s+(a)$. Using the exact sequence
\[ 0 \rar ((a):b)/(a) \rar \RR/(a) \stackrel{\cdot b}{\rar}  \RR/(a) \rar  \RR/\C \rar 0,\]
 we get
 \[ \lambda(\RR/\C) =  \lambda	( ((a):b)/(a) ) = \lambda(\H_{1}(\C)).\]
Now, using the exact sequence
\[ 0 \rar \H_{1}(\C) \rar (\RR/\C)^{2} \rar \C/\C^{2} \rar 0, \]
we get
\[ \lambda(\C/\C^{2}) = 2 \lambda(\RR/\C) - \lambda(\H_{1}(\C)) = \lambda(\RR/\C).\]
Hence,
\[ \lambda(\C^{2}/a \C) = \lambda(\C/a \C) - \lambda(\C/ \C^{2}) = \lambda(\C/a \C) - \lambda(\RR/ \C) = \lambda(\C/a \C) - \lambda((a)/a \C) = \lambda(\C/(a)). \qedhere \]
\end{proof}

\begin{Example}{\rm Let $H=\left<a,b,c \right>$ be a numerical semigroup which is minimally generated by positive integers $a,b,c$ with gcd$(a,b,c)=1$. If the semigroup ring $\RR=k[[t^a,t^b,t^c]]$ is not a Gorenstein ring, then $r(\RR)=2$ (see \cite[Section 4]{GMP11}).}
\end{Example}

\begin{Example}\label{Ex1linkage}{\rm Let $A=k[X, Y, Z]$, let
$I=(X^2-YZ, Y^2-XZ, Z^2-XY)$ and $\RR=A/I$. Let $x, y, z$ be the images of $X, Y, Z$ in $\RR$. By \cite[Theorem 10.6.5]{AusBr}, we see that $\C=(x, z)$ is a canonical ideal of $\RR$ with a minimal reduction $(x)$. It is easy to see that $\rho(\RR)=2$, $\e_1(\C)=2$ and $\cdeg(\RR)=1$.}
\end{Example}

\begin{Example}{\rm Let $A=k[X,Y,Z]$, let $I=(X^4-Y^{2}Z^{2}, Y^{4}-X^{2}Z^{2}, Z^{4}-X^{2}Y^{2})$ and $\RR=A/I$. Let $x, y, z$ be the images of $X, Y, Z$ in $\RR$. Then $\C=(x^{2}, z^{2})$ is a canonical ideal of $\RR$ with a minimal reduction $(x^{2})$. We have that $\rho(\RR)=2$, $\e_1(\C)=16$ and $\cdeg(\RR)=8$.}
\end{Example}

\subsubsection*{Lower and upper bounds for the canonical index}

\begin{Example} {\rm Let $e \ge 4$ be an integer and let $H = \left<e, \{e+i\}_{3 \le i \le e-1}, 3e+1, 3e+2\right>$.
Let $k$ be a field and $V = k[[t]]$  the formal power series ring over $k$. Consider the semigroup ring $\RR = k[[H]] \subseteq V$.
\begin{enumerate}[(1)]
\item The conductor of $H$ is  $c= 2e+3$.

\smallskip

\item The canonical module is $K_{\RR} = \left<1, t\right>$ and $K_{\RR}^{e-2} \subsetneq K_{\RR}^{e-1} = V$.

\smallskip

\item The  canonical ideal of $\RR$ is $\C=(t^c K_{\RR})$ and $Q = (t^c)$ is a minimal reduction of $\C$. Moreover, $\rho(\RR)= \red(\C)= e-1$.

\smallskip

\item The canonical degree is $\cdeg(\RR)= \lambda(\C/Q) = \lambda(K_{\RR}/\RR) = 3$.

\smallskip

\item In particular, $\cdeg(\RR) \leq \rho(\RR)$.
\end{enumerate}
 }\end{Example}

\begin{Theorem}\label{SallyofC} Let $ (\RR, \m)$ be a $1$-dimensional Cohen-Macaulay local ring with a canonical ideal $\C$. Suppose that the type of $\RR$ is $2$.
Then we have the following.
\begin{enumerate}[{\rm (1)}]
\item ${\ds \e_{1}(\C) \leq \rho(\RR) \, \cdeg(\RR)}$.

\smallskip

\item $\rho(\RR)=2$ if and only if ${\ds \e_{1}(\C) = 2 \, \cdeg(\RR)}$.
\end{enumerate}
\end{Theorem}

\begin{proof} Let $\C = (a, b)$, where $(a)$ is a minimal reduction of $\C$.

\medskip

\noindent (1) For each $j=0, \ldots, \rho(\RR) -1$,  the module $\C^{j+1}/a\C^{j}$ is cyclic and annihilated by $L=\ann(\C/(a))$.  Hence we obtain
\[\e_1(\C) = \sum_{j=0}^{\rho(\RR) -1} \lambda(\C^{j+1}/a\C^j)  \leq  \rho(\RR) \, \lambda(\RR/L)= \rho(\RR)  \, \cdeg(\RR).\]
(2) Note that $\rho(\RR)=2$ if and only if ${\ds \e_{1}(\C) = \sum_{j=0}^{1} \lambda(\C^{j+1}/a\C^j) }$.
Since $\nu(\C)= r(\RR)=2$, by Proposition~\ref{C2}, $\lambda(\C/(a) ) = \lambda(\C^{2}/a \C)$.  Thus, the assertion follows from
\[  \sum_{j=0}^{1} \lambda(\C^{j+1}/a\C^j) = 2 \, \lambda(C/(a)) = 2 \, \cdeg(\RR). \qedhere \]
\end{proof}

 \begin{Proposition}
 Let $(\RR, \m)$ be a Cohen-Macaulay local ring with infinite residue field and a canonical ideal $\C$. Suppose that $\RR_\p$ is a Gorenstein ring for $\forall \p \in \Spec (\RR) \setminus \{\m\}$ and that $\C$ is equimultiple.
\begin{enumerate}[{\rm (1)}]
\item $\C^n$ has finite local cohomology for all $n > 0$.

\smallskip

\item Let $\Bbb I (\C^n)$ denote the Buchsbaum invariant of $\C^n$. Then the nonnegative integer
$ \beta(\RR) =  \sup_{n > 0}\Bbb I (\C^n) $ is independent of the choice of $\C$.

\smallskip

\item $\rho(\RR) \leq \deg(\RR) + \beta(\RR) -1$.
\end{enumerate}
 \end{Proposition}

 \begin{proof}
 (1) The assertion follows from $\C^n\RR_\p = a^n\RR_\p$ where $Q = (a)$ is a reduction of $\C$.

 \medskip

 \noindent (2)  We have $\C^n \cong \mathcal{D}^n$ for any canonical ideal $\mathcal{D}$ and $\C^{n+1} = a \C^n$ for $\forall n \gg 0$.

 \medskip

 \noindent (3) Let $\q$ be a minimal reduction of $\m$. Then, for $\forall n > 0$, we have
\[ \nu(\C^n) = \lambda(\C^n/\m \C^n) \leq \lambda(\C^n/\q \C^n) \leq \e_0(\q,\C^n) + \Bbb I (\C^n) \leq  \deg(\RR) + \beta(\RR). \] 
The conclusion follows from \cite[Theorem 1]{ES}.
\end{proof}

\begin{Example}{\rm  Consider the following examples of $1$-dimensional Cohen-Macaulay semigroup rings $\RR$ with a canonical ideal $\C$ such that $\cdeg(\RR) =r(\RR)$.
\begin{enumerate}[(1)]
\item Let $a\geq 4$ be an integer and let $H= \left<a, a+3, \ldots, 2a-1, 2a+1, 2a+2 \right>$.  Let $\RR=k[[H]]$. Then the canonical module of $\RR$ is
$ \omega = \left< 1, t, t^3, t^4, \ldots, t^{a-1} \right>$.
The ideal $Q=(t^{a+3})$  is a minimal reduction of $\C= (t^{a+3}, t^{a+4}, t^{a+6}, t^{a+7}, \ldots, t^{2a+2} )$,
\[ \cdeg(\RR)=a-1=\nu(\C), \quad \mbox{and} \quad \red (\C)=2.\]

\smallskip

\item Let $a\geq 5$ be an integer and let $H= \left<a, a+1, a+4, \ldots, 2a-1, 2a+2, 2a+3 \right>$. Let $\RR=k[[H]]$. Then the canonical module of $\RR$ is
 $\omega = \left< 1, t, t^4, t^5, \ldots, t^{a-1} \right>$. The ideal $Q=(t^{a+4})$ is a minimal reduction of 
$\C= (t^{a+4}, t^{a+5}, t^{a+8}, t^{a+9}, \ldots, t^{2a+3})$.
\[ \cdeg(\RR)=a-2=\nu(\C), \quad \mbox{and} \quad \red(\C)=3.\]
\end{enumerate}
}\end{Example}

\section{Roots of canonical ideals}
\noindent
 Another phenomenon concerning canonical ideals  is the
 following.
 
 \medskip

\begin{Definition}{\rm
Let $(\RR,\m)$ be a Cohen-Macaulay local ring of dimension $d \geq 1$ with a canonical ideal $\C$.
An ideal $L$ is called a {\em root} of $\C$ if $L^{n} \simeq \C$ for some $n$.
In this case, we write $\tau_{L}(\C) = \min \{ n \mid L^{n} \simeq \C \}$.
Then the {\em rootset} of $\RR$ is the set  $\mathrm{root}(\RR) = \{ \tau_{L}(\C) \mid L \; \mbox{is a root of} \; \C \} $.
} \end{Definition}

The terminology  `roots' of $\C$ already appears in \cite{BV1}.
  Here is  a simple example.

 \begin{Example}{\rm (\cite[Example 3.4]{BV1})} \label{ex1root}
 {\rm Let $\RR=k[[t^4, t^5, t^6, t^7]]$. Then $\C= (t^4, t^5, t^6)$ is a canonical ideal of $\RR$.
 Let $I = (t^4, t^5)$. Then  $I^2 = t^4\C$, that is, $I$ is a square root of $\C$.
 }\end{Example}

It was studied in \cite{blue1} for  its role in examining properties
of canonical ideals. Here
is one instance:

\begin{Proposition}
Let $\RR$ be a 
Cohen-Macaulay local ring of dimension one. Let $L$ be a root of $\C$.
If $L$ is an irreducible 
 ideal 
 then $\RR$ is a Gorenstein ring.
\end{Proposition}

\begin{proof} Note that if $q\in Q$ satisfies $q L \subset L$ then $qL^n \subset L^n$. This implies that $q\C \subset \C$ and so
$q\in \RR$. By Proposition~\ref{recog}, $L \simeq \C$.   We now make use of a technique of \cite{blue1}. From $\C \simeq \C^n $ we have 
\[\C \simeq \C^{n-1} \cdot \C \simeq \C^{n-1} \cdot \C^n = \C^{2n-1}.\]
By iteration we get
that $\C \simeq \C^m$ for arbitrarily large values of $m$, and for all of them we have $\C^m:\C^m = \RR$. 

\medskip

\noindent We may assume that the residue
field of $\RR$ is infinite
and obtain a minimal reduction $(c)$ for $\C$, that is an equality $\C^{r+1} = c\C^r$ for all $r\geq s$ for some $s$. This gives
$(\C^r)^2 = c^r \C^r$, and therefore $(\C^r c^{-r}) \C^r = \C^r$. Since
$\C^r:\C^r = \RR$ the equality gives that $\C^r \subset (c^r) \subset \C^r$  and therefore $\C^r = (c)^r$. Thus $\C$ is invertible, as desired.
\end{proof}

\begin{Question}{\rm 
Suppose there is an ideal $L$ such that $L^2 = \C$. 
Sometimes they imply that $\cdeg(\RR)$ is even; can it be the same with $\ddeg(\RR)$?
}\end{Question}

\begin{Question}{\rm 
If $L^n \simeq \C$ and $L$ is reflexive must  $\RR$ be Gorenstein?
}\end{Question}

\begin{Exercise}{\rm Let $(\RR, \m)$ be a Cohen-Macaulay local ring of dimension $d\geq 2$. Let $\C$ be a Cohen-Macaulay ideal of 
codimension $1$. A goal is to derive
  a criterion for $\C$ to be a canonical ideal. Let $x$ be regular mod $\C$. If $\C/x\C$ is a canonical 
ideal of $\RR/x\RR$ then $\C$  is canonical ideal of $\RR$. We make use of two facts:
\begin{enumerate}[(1)]
\item $\C$ is a canonical module iff $\Ext_{\RR}^j (\RR/\m, \C) = \RR/\m$ for $j=d$ and $0$ otherwise (\cite[Theorem 6.1]{HK2}).

\smallskip

\item The change of rings equation asserts $\Ext_{\RR}^j(\RR/\m, \C) \simeq \Ext_{\RR/(x)}^{j-1} (\RR/\m, \C/x\C)$, $j\leq d$.
\end{enumerate}
More generally, how do we tell when a Cohen-Macaulay ideal $I$ is irreducible?
}\end{Exercise}


 The set  $\mathrm{root}(\RR)$ is clearly independent of the chosen canonical ideal.
 To make clear the character of this set we appeal to the following property. We call an ideal of $\RR$ closed if $\Hom(L, L) =\RR$.
 
    \begin{Proposition}\label{root0}
Let $(\RR,\m)$ be a $1$-dimensional Cohen-Macaulay local ring with a canonical ideal $\C$.
Let $L$ be a root of $\C$.
 If $\Hom(L^{n}, L^{n}) \simeq \RR$ for infinitely many values of $n$ then
 all powers of $L$ are closed. In this case $L$ is invertible and $\RR $ is Gorenstein.
 \end{Proposition}

\begin{proof}
 A property of roots is that they are closed ideals. More generally, it is clear that
  if $\Hom(L^m, L^m) = \RR$ then $\Hom(L^n, L^n)=\RR$ for $n<m$, which shows the first assertion.

\medskip

 We may assume that $\RR$ has an infinite residue field.
Let $s=\red(L)$. Then $L^{s+1} = xL^s$, for a minimal reduction $(x)$ and thus
$ L^{2s} = x^sL^s$, which gives that
\[ x^{-s}L^s \subset \Hom(L^s,L^s) \simeq \RR.\]
It follows that $L^s \subset (x^s) \subset L^s$, and thus $L^s = (x^s)$. Now taking  $t$ such that $L^t\simeq C$ shows that $C$ is principal.
\end{proof}

\begin{Corollary} \label{uniqueroot} 
 If $L^m \simeq \C \simeq L^n$, for $m\neq n$, then $C$ is principal.
\end{Corollary}
\begin{proof}
 Suppose $m > n$. Then
\[ \C \simeq L^m = L^n L^{m-n} \simeq L^m L^{m-n} = L^{2m -n}.\] Iterating,
 $L$ is a root of $\C$ of arbitrarily high order.
 \end{proof}

\begin{Corollary}{\rm (\cite[Proposition 3.8]{BV1})}
Let $(\RR,\m)$ be a $1$-dimensional Cohen-Macaulay local ring with a canonical ideal $\C$.
If $\RR$ is not Gorenstein then no proper power of $\C$ is a canonical ideal.
 \end{Corollary}

 \begin{Proposition}\label{root1} \index{rootset; finiteness}
 Let $(\RR,\m)$ be a Cohen-Macaulay local ring of dimension $d \geq 1$ with infinite residue field and with an equimultiple canonical ideal $\C$.
 Let $L$ be a root of $\C$. Then  $\tau_{L}(\C)  \leq \min\{ r(\RR)-1, \red(L)\}$.
 \end{Proposition}

\begin{proof} Suppose   $n = \tau_{L}(\C) \geq r(\RR)$. Then
\[  \nu(L^{n}) = \nu (\C) = r(\RR) < n+1 =  {{n+1}\choose{1}}.\]
By \cite[Theorem 1]{ES}, there exists a reduction $(a)$ of $L$ such that
$L^{n} = aL^{n-1}$. Thus, $L^{n-1} \simeq \C$, which contradicts the minimality of $\tau_{L}(\C)$.
\end{proof}

\begin{Remark}{\rm  We have the following.

\begin{enumerate}[(1)]
\item If $r(\RR)=2$, then the isomorphism class of $\C$ is the only root.

\smallskip

\item The upper bound in Proposition~\ref{root1} is sharp. For example, let $a \ge 3$ be an integer and we consider the numerical semigroup ring $\RR = k[[\{t^i\}_{a \le i \le 2a-1}]] \subseteq k[[t]]$. Then the canonical module of $\RR$ is
\[ \omega = \sum_{i=0}^{a-2} \RR t^i = (\RR + \RR t)^{a-2}.\] Thus $\RR$ has a canonical ideal $\C= t^{a(a-2)} \omega$.  Let $L = (t^a,t^{a+1})$. Then $\C = L^{a-2}$. Hence we have   $ \tau_{L}(\C) = a-2  = r (\RR) -1$.
\end{enumerate}
}\end{Remark}

Combining Propositions \ref{root0} and \ref{root1} we have the following.

\begin{Theorem}\label{root}
Let $(\RR,\m)$ be a $1$-dimensional Cohen-Macaulay local ring with a canonical ideal. If $\RR$ is not Gorenstein, then
$\mathrm{root}(\RR)$ is a finite set of cardinality less than $r(\RR)$.
\end{Theorem}

\subsubsection*{Applications of roots}

\begin{Proposition}
Let $(\RR,\m)$ be a $1$-dimensional Cohen-Macaulay local ring with a canonical ideal $\C$.
Let $f$ be the supremum of the reduction numbers of the $\m$-primary ideals. Suppose that $L^{n} \simeq \C$.
If $p$ divides $n$, then $\rho(\RR) \leq (f+p-1)/p$.
\end{Proposition}

\begin{proof}
 Since $n=pm$, by replacing $L^m$ by $I$, we may assume that $I^p = \C$. Let $r=\red(I)$ with $I^{r+1} = xI^r$. Then $r=ps+q$ for some $q$ such that $-p+1 \leq q \leq 0$.
Since $ps=r-q \geq r$, we have
\[ \C^{s+1} = I^{ps+p} = x^{p} I^{ps} = x^{p} \C^{s}.\]
Thus, $ \rho(\RR)= \red(\C) \leq s = (r-q)/p \leq (r+p-1)/p \leq (f+p-1)/p$.
\end{proof}

\subsubsection*{A computation of roots of the canonical ideal}
 
Let $0 < a_1< a_2 < \cdots < a_q$ be integers such that $\operatorname{gcd}(a_1, a_2, \ldots, a_q) = 1$. Let $H = \left<a_1, a_2, \ldots, a_q\right>$ be the numerical semigroup generated by $a_i's$.
 Let $V = k[[t]]$ be the formal power series ring over a field $k$ and set $\RR = k[[t^{a_1},t^{a_2}, \ldots, t^{a_q}]]$. We denote by $\m$ the maximal ideal of $\RR$ and by $e = a_1$ the multiplicity of $R$. Let $v$ be the discrete valuation of $V$. 
In what follows, let $\RR \subseteq L \subseteq V$ be a finitely generated $\RR$-submodule of $V$ such that $\nu(L)>1$. We set $\ell = \nu(L)-1$. Then we have the following.

\begin{Lemma}\label{lemma1}
With notation as above, $1 \not\in \m L$.
\end{Lemma}

\begin{proof}
Choose $0 \ne g \in \m$ so that $gV \subsetneq \RR$. Then $Q = g\RR$ is a minimal reduction of the $\m$-primary ideal $I = gL$ of $\RR$, so that $g \not\in \m I$. Hence $1 \notin\m L$.
\end{proof}

\begin{Lemma}\label{2}
With notation as above, there exist elements $f_1, f_2, \ldots, f_\ell \in L$ such that 
\begin{enumerate}[{\rm (1)}]
\item $L = \RR + \sum_{i=1}^\ell \RR f_i$,
\smallskip
\item $0 < v(f_1) < v(f_2) < \ldots <v(f_\ell)$, and
\smallskip
\item $v(f_i)\not\in H$ for all $1 \le i \le \ell$.
\end{enumerate}
\end{Lemma}

\begin{proof}
Let $L = \RR + \sum_{i=1}^\ell \RR f_i$ with $f_i \in L$. Let $1 \le i \le \ell$ and assume that $m= v(f_i) \in H$. We write $f_i = \sum_{j = m}^\infty c_jt^j$ with $c_j \in k$. Then $c_s \ne 0$ for some $s > m$ such that $s \not\in H$, because $f_i \not\in \RR$. Choose such integer $s$ as small as possible and set $h = f_i - \sum_{j=m}^{s-1}c_jt^j$. Then $\sum_{j=m}^{s-1}c_jt^j \in \RR$ and $\RR + \RR f_i = \RR + \RR h$. Consequently, as  $v(h) =s  > m = v(f_i)$, replacing $f_i$ with $h$, we may assume that $v(f_i) \not\in H$ for all $1 \le i \le \ell$. Let $1 \le i<j \le \ell$ and assume that $v(f_i) = v(f_j)=m$. Then, since $f_j = cf_i + h$ for some $0 \ne c \in k$ and $h \in L$ such that  $v(h) >m$, replacing $f_j$ with $h$, we may assume that $v(f_j) > v(f_i)$. Therefore we can choose a minimal system of generators of $L$ satisfying conditions (2) and (3). 
\end{proof}

\begin{Proposition}\label{NS1} With notation as above, let $1, f_1, f_2, \ldots, f_\ell \in L$  and $1, g_1, g_2, \ldots, g_\ell \in L$ be systems of generators of $L$ and assume that both of them satisfy conditions $(2)$ and $(3)$ in {\rm Lemma~\ref{2}}. Suppose that $v(f_\ell) < e=a_1$. Then $v(f_i) = v(g_i)$ for all $1 \le i \le \ell$.
\end{Proposition}

\begin{proof} We set $m_i = v(f_i)$ and $n_i = v(g_i)$ for each $1 \le i \le \ell$. 
Let us write
\begin{eqnarray*}
f_1&=& \alpha_0 + \alpha_1 g_1 + \ldots + \alpha_\ell g_\ell\\
g_1&=& \beta_0 + \beta_1 f_1 + \ldots + \beta_\ell f_\ell
\end{eqnarray*}
with $\alpha_i, \beta_i \in \RR$. Then $v(\beta_1 f_1 + \ldots + \beta_\ell f_\ell) \ge v(f_1) = m_1>0$, whence $\beta_0 \in \m$ because $v(g_1)=n_1 > 0$. We similarly have that $\alpha_0 \in \m.$ Therefore $n_1 = v(g_1) \ge m_1$, since $v(\beta_0) \ge e > m_\ell \ge m_1$ and $v(\beta_1 f_1 + \ldots + \beta_\ell f_\ell) \ge m_1$. Suppose that $n_1 > m_1$. Then  $v(\alpha_1 g_1 + \ldots + \alpha_\ell g_\ell) \ge n_1 > m_1$ and $v(\alpha_0) \ge e > m_1$, whence $v(f_1)>m_1$, a contradiction. Thus $m_1 = n_1$. 

\medskip

\noindent Now let $1 \le i < \ell$ and assume that $m_j = n_j $ for all $1 \le j \le i$. We want to show $m_{i+1} = n_{i+1}$. Let us write
\begin{eqnarray*}
f_{i+1}&=& \gamma_0 + \gamma_1 g_1 + \ldots + \gamma_\ell g_\ell\\
g_{i+1}&=& \delta_0 + \delta_1 f_1 + \ldots + \delta_\ell f_\ell
\end{eqnarray*}
with $\gamma_i, \delta_i \in \RR$. 

\medskip

\noindent First we claim that $\gamma_j, \delta_j \in \m$ for all $0 \le j \le i$. 
 As above, we get $\gamma_0, \delta_0 \in \m$.  Let $0 \le k < i$ and assume that $\gamma_j, \delta_j \in \m$ for all $0 \le j \le k$. We show $\gamma_{k+1}, \delta_{k+1} \in \m$. Suppose that $\delta_{k+1} \not\in \m$. Then as $v(\delta_0+ \delta_1 f_1 + \ldots + \delta_k f_k) \ge e > m_{k+1}$,  we get  $v(\delta_0 + \delta_1 f_1 + \ldots + \delta_k f_k + \delta_{k+1} f_{k+1})=m_{k+1}$, so that $n_{i+1} = v(g_{i+1})=v(\delta_0+ \delta_1 f_1 + \ldots + \delta_\ell f_\ell) = m_{k+1}$, since $v(f_h) =m_h> m_{k+1}$ if $h> k+1$. This is impossible, since $n_{i+1} > n_i = m_i \ge m_{k+1}$. Thus $\delta_{k+1} \in \m$. We similarly get $\gamma_{k+1} \in \m$, and the claim is proved.

\medskip

\noindent Consequently, since $v(\delta_0 + \delta_1 f_1 + \ldots + \delta_i f_i) \ge e > m_{i+1}$ and $v(f_h) \ge  m_{i+1}$ if $h\ge i+1$, we have $n_{i+1} = v(g_{i+1}) = v(\delta_0 + \delta_1 f_1 + \ldots + \delta_\ell f_\ell)\ge m_{i+1}$. Assume that $n_{i+1} > m_{i+1}$. Then since $v(\gamma_0 + \gamma_1 g_1 + \ldots + \gamma_i g_i) \ge e > m_{i+1}$ and $v(g_h) \ge  n_{i+1} > m_{i+1}$ if $h\ge i+1$, we have $m_{i+1} = v(f_{i+1}) = v(\gamma_0 + \gamma_1 g_1 + \ldots + \gamma_\ell g_\ell) > m_{i+1}$. This is a contradiction. Hence $m_{i+1} = n_{i+1}$, as desired.
\end{proof}

 \begin{Theorem}\label{th1-1}
With notation as above, assume that $r (\RR) = 3$ and write the canonical module $\omega = \left<1, t^a,t^b\right>$ with $0 < a < b$. Suppose that $2a > b$ and $b < e$. Let $L$ be an ideal of $\RR$ and let $n \geq 1$ be an integer. If $L^n \cong \omega$, then $n = 1$.
\end{Theorem}

\begin{proof}
By Proposition~\ref{root1}, we have $n \leq 2$. Suppose that $n=2$. Let $f \in L$ such that $fV = LV$ and set $M= f^{-1}L $. Then $\RR \subseteq M \subseteq V$ and $M^2 = f^{-2}L^{2} \cong \omega$. By Lemma \ref{2}, we can write $M = \left<1, f_1, f_2, \ldots, f_{\sigma} \right>$,  where $0 < v(f_1) < v(f_2) < \cdots < v(f_{\sigma} )$ and $v(f_i) \not\in H$ for every $1 \leq i \leq \sigma$. Then $M^2 = \left<1, \{f_i\}_{1 \leq i \leq \sigma}, \{f_if_j\}_{1 \leq i \leq j \leq \sigma} \right>$. Let $n_i = v(f_i)$ for each $1 \leq i \leq \sigma$.

\medskip

\noindent Suppose that $\sigma > 1$. Then we claim that $f_1 \not\in \left<1, \{f_i\}_{2 \leq i \leq \sigma}, \{f_if_j\}_{1 \leq i \leq j \leq \sigma} \right>$.
In order to prove the claim, we assume the contrary and write $$f_1 = \alpha + \sum_{i=2}^{\sigma}\alpha_i f_i + \sum_{1 \leq i \leq j \leq \sigma}\alpha_{ij}f_if_j$$ with $\alpha, \alpha_i, \alpha_{ij} \in \RR.$ Then since $n_1 < n_i$ for $i \geq 2$ and $n_1 \leq n_i < n_i+n_j$ for $1 \leq i \leq j \leq \sigma$, we have $n_1 = v(\alpha)$, which is impossible, because $n_1 \not\in H$ but $v(\alpha)\in H$.

\medskip

\noindent Since $\nu(M^2) = 3$, by the claim we have $M^2=\left<1, f_1, f_if_j\right>$ for some $1 \leq i \leq j \leq \sigma$. In fact, the other possibility is $M^2 = \left<1,f_1,f_i\right>$ with $i \geq 2$. However, when this is the case, we get $M= M^2$ so that $\red_{(f)}L \leq 1$. Thus, $\red_{(f^2)}L^2 \leq 1$. Therefore, since $L^2 \cong \omega$, by Remark \ref{R} (1), $\RR$ is a Gorenstein ring, which is a contradiction.

\medskip

\noindent We now choose $0 \neq \theta \in \mathrm{Q}(V)$, where $\mathrm{Q}(V)$ is the quotient field of $V$, so that $\omega : M^2 = \RR \theta$. Then $\omega = \theta M^2$ and hence $\theta$ is a unit of $V$. We compare
\[ \omega = \left<1, t^a, t^b \right> = \left<\theta, \theta f_1, \theta f_if_j\right>\]
and notice that $0 < a < b < e$ and $0< n_1 < n_i + n_j$. Then by Proposition~\ref{NS1}, we get that $n_1 = a$ and $n_i + n_j = b$, whence $b \geq 2a$. This is a contradiction.
\end{proof}

Let us give examples satisfying the conditions stated in Theorem \ref{th1-1}.

\begin{Example}{\rm Let $e \geq 7$ be an integer and set
\[ H = \left<e + i \mid 0 \le i \le e-2 ~\text{such~that}~ i \ne e-4, e-3 \right>.\]
 Then $\omega = \left<1, t^2, t^3 \right>$ and $r (\RR) =3$. More generally, let $a,b, e \in \Bbb Z$ such that $0 < a< b$, $b < 2a$, and $e \geq a+b+2$. We consider the numerical semigroup
 \[ H = \left< e+i \mid 0 \leq i \leq e-2~\text{such~that}~ i \neq e-b-1, e-a-1\right>.\] Then $\omega =\left <1, t^a, t^b \right>$ and $r (\RR) = 3$. These rings $\RR$ contain no ideals $L$ such that $L^n \cong \omega$ for some integer $n \geq 2$.
}\end{Example}

\section{Minimal values of bi-canonical degrees}
\noindent
Now we begin to examine the significance of the values of $\ddeg(\RR)$. We focus on rings of dimension $1$.

\subsubsection*{Almost Gorenstein rings} \label{aGor} 

First we recall the definition of almost Gorenstein rings 
 (\cite{BF97, GMP11, GTT15}).

\begin{Definition}{\rm (\cite[Definition 3.3]{GTT15})  A Cohen-Macaulay local ring $\RR$ with a canonical module $\omega$  is said to be an {\em 
almost Gorenstein} ring if there exists an exact sequence of $\RR$-modules
${\ds 0\rightarrow\RR\rightarrow \omega \rightarrow X\rightarrow  0}$
such that $\nu(X)=\e_0(X)$. In particular if $\RR$ has dimension one $X = (\RR/\m)^{r-1}$, $r = r(\RR)$.
}\end{Definition}

\begin{Theorem}\label{AGorddeg} Let $(\RR,\m)$ be a Cohen-Macaulay  local ring of dimension $1$
 with a canonical ideal $\C$.  If $\RR$ is almost Gorenstein
then 
 $\ddeg(\RR) =1$.
\end{Theorem}

\begin{proof}
In dimension $1$,
 from
 \[
  0 \rar (c) \lar \C \lar X \rar 0, \]
$X$ is a vector space $k^{r-1}$, $r=r(\RR)$.
 To determine
 $\C^{**}$ apply
  $\Hom(\cdot, (c))$ to the above exact sequence to get
$\Hom(\C, (c)) = \m$. On the other hand, by Proposition~\ref{bidual}, $\C^{**} = \Hom(\m, (c))=L_0$,   the socle of $\RR/(c)$ [which
is generated by $r$ elements], 
 properly containing $\C$  that is
$\C^{**} = L$, the socle of $\C$. Therefore $\ddeg(\RR) = \lambda(L/C) = 1$.
\end{proof}

The example below shows that the
converse does not holds true.

\begin{Example}{\rm Consider the monomial ring (called to our attention by Shiro Goto)
  $\RR = \mathbb{Q}[t^5, t^7, t^9], 
 \m = (x,y, z) $. We have a presentation $\RR = \mathbb{Q}[x,y,z]/P$, with $P = (y^2-xz, x^5-yz^2, z^3-x^4 y)$. Let
  us examine some properties of $\RR$. For simplicity we denote the images of $x,y,z$ in $\RR$ 
   by the  same symbols.
  An explanation for these calculations can be found in the proof of Theorem~\ref{preaGor}. 

\begin{enumerate}[(1)]
\item Let $\C = (x,y)$. A calculation with {\em Macaulay2} shows that if $\D = \C:\m$, then $\lambda(\D/\C) = 1$. Therefore
$\C$ is a canonical ideal by  Corollary~\ref{recogcor}.  

\smallskip

\item $(c):\C \neq \m$ so $\RR$ is not almost Gorenstein. However $\C^{**} = (c):  [(c):\C]$ satisfies  [by another {\em Macaulay2} calculation]
$\lambda(\C^{**}/\C) = 1$, so $\C^{**} = L$. This shows that $\ddeg(\RR) = 1$.

\smallskip

\item This example shows that $\ddeg(\RR) = 1$ holds [for dimension one] in a larger class rings than almost Gorenstein rings.

\smallskip

\item Find $\red(\C)$  and $\red(\D)$ for this example. Need the minimal reductions. 

\end{enumerate}
 }\end{Example} 

\subsubsection*{Goto rings} We now examine the significance of a minimal value for $\ddeg(\RR)$. Suppose $\RR$ is not a Gorenstein ring.

\begin{Definition}\label{Gotoring} {\rm  A Cohen-Macaulay local ring $\RR$ of dimension $d$ is a {\em Goto ring} if it has a canonical
ideal and $\ddeg(\RR) =1$.
}\end{Definition}

\begin{Questions}{\rm
\begin{enumerate}[(1)]

\item  What other terminology  should  be used?    Nearly
 Gorenstein ring  has already been used in \cite{Herzog16}. We will examine its relationship to Goto rings.

\smallskip

\item  Almost Gorenstein rings of dimension one have $\red(G)=2$. What  about these rings?

\smallskip

\item 
 What are the properties of the Cohen-Macaulay module $X$
defined by
\[ 0 \rar (c) \lar \C \lar X \rar 0,\]
{\em nearly Ulrich or pre-Ulrich bundles}?

\smallskip

\item $\lambda(\C^{**}/\C)=1$ implies that
\[ \C^{**}/\C = L/\C= \C:\m/\\C\simeq \RR/\m.\]

\item If $L_0= (c):\m$, the socle of $(c)$, is equal to  $L$, then $\C/(c)$ is a vector space, so $\RR$ is almost
Gorenstein and conversely.

\smallskip

\item In all cases
$L^2 = \C L$, $\m L = \m \C$, so $ \C/\m \C \hookrightarrow 
L/\m L $ (\cite[Theorem 3.7]{CHV}). Therefore if $\RR$ is not almost Gorenstein, 
  $c$ cannot be a minimal generator of $L$ and thus
  $L = (\C, \alpha)$, $\nu(L) = r+1$, with $\alpha\in L_0$, or $L = (L_0, \beta)$, 
  with $\beta \in \C$.
  
  \smallskip
  
  \item
  This says that 
  \begin{eqnarray*}
  L &=& (x_1, \ldots, x_r, \beta), \quad x_i \in (c):\m,\\
  \C & = & (c,  x_2, \ldots, x_{r-2}, \beta),  \quad c \in \m L\\
  \C^{*} & = &
  c^{-1}[(c):\C]    \\
  \C^{**} & = & (c):((c)): \C) = L\\  
  \C^{*} &=& \C^{***} = c^{-1} [(c):L] \\
  (c):\C & = & (c):L = (c): \beta.
  \end{eqnarray*} 
\end{enumerate}
}\end{Questions}

Let us attempt to get $\ddeg(\RR)$ for $\RR$ almost Gorenstein but of dimension $\geq 2$. Note that $\cdeg(\RR) = r-1$. Assume $d=r=2$: can we complete the calculation?

\begin{Proposition}\label{AGorddeg2} Let $(\RR,\m)$ be a Cohen-Macaulay  local ring 
 with a canonical ideal $\C$. 
\begin{enumerate}[{\rm (1)}]

\item From the 
\[ 0 \rar (c) \lar \C \lar X\rar 0\]
applying $\Hom(\cdot, (c))$, we get
\[ 0\rar  \Hom(\C, (c)) \lar \RR \lar  \Ext(X, (c)) \lar \Ext(\C, (c)) \rar 0.\]

\item The image of $\Hom(\C, (c)) $ in $\RR$ is a proper ideal $N$, and so $\RR/N$ is a submodule of $\Ext(X, (c))$ which is annihilated by
$\ann(X)$ $($which contains
$(c)$$)$.  We note $($see \cite[p. 155]{Kapbook}$)$ that
\begin{eqnarray*}
\Ext(X, (c))  \simeq
& \Ext^1_{\RR}(X, \RR) = \Hom_{\RR/(c))} (X, \RR/(c)) = ((c):N)/(c), \quad \mbox{\rm if } r=2.
\end{eqnarray*}

\smallskip

\item Suppose $\C$ is equimultiple  and $\ddeg(\RR) = 1$. Then $\RR$ is Gorenstein at all  primes of codimension one
 with one exception, call it $\p$. This means 
that $\C^{**}/\C$ is $\p$-primary and 
\[ \deg(\C^{**}/\C)= \ddeg(\RR_\p)
\cdot \deg(\RR/\p)=1.\]
\end{enumerate}
\end{Proposition}

\begin{Remark}{\rm 
What sort of modules are $\C/(c)$ and $\C^{**}/\C$? The first we know is Cohen-Macaulay, how about the second?

\begin{enumerate}[(1)]

\item  $\C^{**} = \C:\p$: both are divisorial ideals that agree in codimension one.

\smallskip

\item If $\dim \RR=2$, $C^{**}/\C$ is Cohen-Macaulay of dimension one and multiplicity one so what sort of module is it? It is an $\RR/\p$-module
of rank one and $\RR/\p$ is a discrete valuation domain
 so $\C^{**}/\C \simeq \RR/\p$. 

\smallskip

\item In all dimensions, $\C^{**}/\C$ has only $\p$ for associated prime and is a torsion-free $\RR/\p$-module of rank one. Thus it is 
isomorphic to an ideal of $\RR/\p$.  
$\C^{**}/\C$ has  also the condition $S_2$ of Serre from the exact sequence
\[ 0 \rar \C^{**}/\C \lar \RR/\C \lar \RR/\C^{**} \rar 0,\] 
 $\RR/\C^{**}$ has the condition $S_1$.
 If $\RR$ is complete and contains a field by \cite{Nagata} $\RR/\p$ is a regular
local ring and therefore $\C^{**}/\C \simeq \RR/\p$.

\smallskip

\item Some of these properties are stable under many changes of the rings. Will check for generic hyperplane section soon.

\end{enumerate}
}\end{Remark}

\begin{Questions}{\rm
 \begin{enumerate}[(1)]
\item How to pass from $\ddeg(\AA)$ of a graded algebra $\AA$ to $\ddeg(\BB)$ of one of its Veronese subalgebras? 

  \smallskip
\item If $\SS$ is a finite injective extension of $\RR$, is $\ddeg(\SS) \leq [\SS:\RR]  \cdot \ddeg(\RR)$?
\end{enumerate}
}\end{Questions}

\section{Change of rings} 
\noindent
Let $\varphi:\RR \rar \SS$ be a homomorphism of Cohen-Macaulay rings. We examine
a few cases of the relationship between $\cdeg(\RR)$ and $\cdeg(\SS)$ induced by $\varphi$. We skip polynomial, power series and
completion since flatness makes it straightforward.

\medskip

\subsubsection*{Finite extensions}
If $\RR \rar \SS$ is a finite injective homomorphism of Cohen-Macaulay rings 
and $\C$ is a canonical ideal of $\RR$, then $\D = \Hom(\SS, \C)$ is a canonical module for $\SS$,
according to \cite[Theorem 3.3.7]{BH}.
 Recall how $\SS$ acts on
$\D$: If $f\in \D$ and $a, b\in \SS$, then $a\cdot f(b) = f(ab)$.

\medskip

\subsubsection*{Augmented rings}
A case in point is that of the so-called augmented extensions.
Let $(\RR,\m)$ be an $1$-dimensional Cohen-Macaulay local ring with a canonical ideal. Assume   that $\RR$ is not a valuation domain.
Suppose $\AA$ is  the augmented ring $\RR \ltimes \m$. 
That is, $\AA = \RR\oplus \m \epsilon$, $\epsilon^2 = 0$. [Just to keep the components apart in computations we use $\epsilon$ as a place
holder.]
\medskip

 Let $(c)$ be a minimal reduction of the canonical ideal $\C$. We may assume that $\C\subset \m^2$ by replacing $\C$ by $c\C$ if necessary.
Then a canonical module $\mathcal{D}$ of $\AA$ is
 $\mathcal{D} = \Hom_{\RR} (\AA, \C)$. Let us identify
 $\mathcal{D}$ to an ideal of $\AA$. For this we follow \cite{blue1}.
 
\medskip
 
Let $\RR$ be a commutative ring with total quotient ring $Q$
and let $\mathcal {F}$ denote the set of $\RR$-submodules of $Q$.
 Let $M,K \in \mathcal {F}$. Let $M^\vee = \Hom_\RR(M, K)$ and let $$\AA = \RR \ltimes M$$ denote the idealization of $M$ over $\RR$.
  Then the $\RR$-module $M^\vee \oplus K$  becomes an $\AA$-module
   under the action $$(a,m)\circ (f,x) = (af, f(m)+ax),$$ where $(a,m) \in \AA$ and $(f,x) \in M^\vee \times K$. We notice that the canonical homomorphism 
$$\varphi : \Hom_\RR(\AA,K) \to  M^\vee \times K$$
such that $\varphi(f) = (f\circ \lambda, f(1))$ is an $\AA$-isomorphism, where $\tau  : M \to \AA, \tau (m) = (0,m)$, and that $K:M \in 
\mathcal F$ and $(K:M) \times K \subseteq Q  \ltimes Q$ is an $\AA$-submodule of $Q \ltimes Q$, the idealization of $Q$ over itself. When $Q{\cdot}M = Q$, identifying $\Hom_\RR(M,K)=K:M$, we therefore have a natural isomorphism 
$$\Hom_\RR(\AA,K) \cong (K:M) \times K$$
of $\AA$-modules. 
\medskip

Using this observation, setting $K = \C$, $M = \m$,
 we get the following

 \[ \D =
   \Hom_{\RR}(\RR \oplus \m\epsilon, \C) =   \Hom_{\RR}(\m, \C) \oplus \C\epsilon= L + \C\epsilon
\] because $\C :_Q \m = \C:_{\RR} \m \subset \C:_Q \C = \RR$.
Denote $L=\Hom_{\RR}(\m, \C)$. Then  $L \simeq \C:\m\subset \m$ so that $\mathcal{D}$ is an ideal of $\AA$.

\medskip

Let us   determine $\mathcal{D}^{**}$. The total ring of fractions of $\AA$ is $Q \ltimes Q\epsilon  $. If 
$(a,b \epsilon) \in Q\ltimes Q\epsilon$ is in
\[ \Hom(\mathcal{D}, \AA)= \Hom((\C\oplus  L\epsilon ), (\RR\oplus  \m\epsilon)) , 
\] then
$a \C \subset \RR$, $a L \subset \m $ and $a\C \neq \RR$ as $\RR$ is not Gorenstein. Thus $a\in \C^{*}$. On the other
hand, $b  \C\subset \m$. Thus $b\in \C^{*} $ and conversely. Thus
 \[ \mathcal{D}^{*} = \C^{*} \oplus \C^{*}.\]
Suppose $(a,b\epsilon)\in \D^{**}$,
\[\D^{**} = \Hom((\C^{*} \oplus \C^{*}\epsilon ), (\RR\oplus \m \epsilon )).\]
Then 
$a\C^{*} \subset \RR$ and $\C^{**} \neq \RR$.   
In turn $b\epsilon 
\C^{*} \subset \m \epsilon $ and so $b\in \C^{**} $, and conversely. Thus
\[ \D^{**} = \C^{**} \oplus \C^{**} \epsilon .\]

Let us summarize this calculation.
This gives
\[ \D^{**}/\D = \C^{**}/\C \oplus \C^{**}\epsilon /  L \epsilon .\]

\begin{Proposition} Let $(\RR, \m)$ be a Cohen-Macaulay local ring with a canonical ideal and let 
$\AA = \RR\ltimes \m$. Then
\[ 
\boxed{ \ddeg(\AA) = 2\, \ddeg (\RR) -1.
}\]
In particular if $\RR$ is a Goto ring 
then $\AA$ is also a Goto ring.
\end{Proposition}

By comparison, according to \cite[Theorem 6.7]{blue1}, $\cdeg(\AA) = 2\, \cdeg(\RR) + 2$, so that applying to Example~\ref{C1C2} we get
\[ \cdeg(\AA) = \ddeg(\AA) + 3.\]

\subsubsection*{Products} Let $k$ be a field, and
let $\AA_1, \AA_2$ be two finitely generated Cohen-Macaulay $k$-algebras. Let us look at the canonical degrees of the
{\em product} $\AA = \AA_1 \otimes_k \AA_2$. 

\medskip

\noindent As a rule, if $\BB_i, \CC_i$  are $\AA_i$-modules, we use the natural isomorphism
\[ \Hom(\BB_1\otimes \BB_2, \CC_1\otimes \CC_2) =
 \Hom_{\AA_1\otimes_k \AA_2}(\BB_1\otimes_k \BB_2, \CC_1\otimes_k \CC_2),\]
 which works out to be
 \[ \Hom_{\AA_1} (\BB_1, \CC_1) \otimes_k \Hom_{\AA_2}(\BB_2, \CC_2).\]

\noindent If $\AA_i$, $i=1,2$, are localizations [of fin gen $k$-algebras] then $\BB$ is an appropriate localization. [If $\m_i$ are the maximal ideals of $\AA_i$,
pick primes $M_i$ in $\BB_i$ and a prime $\BB$ over both $\m_i$.]
If $\BB_i$ are finite $\AA_i$-modules
and $\FF_i$ are minimal resolutions over $\AA_i$ [or over $\SS_i$, a localization in the next item], then
$\FF_1 \otimes_k \FF_2$ is a resolution whose
entries lie in appropriate primes. 

\medskip

\noindent  If $\SS_i \rar \AA_i$, $i=1,2$, are presentations of $\AA_i$, $\SS= \SS_1\otimes \SS_2\rar \AA_1\otimes \AA_2 = \AA$
 gives a presentation of $\AA$
  and from it we gather the 
invariants [all $\otimes$ over $k$].

\begin{Proposition} Let $\AA_i$, $i=1,2$, be as above. Then
\begin{enumerate}[{\rm (1)}]

\item $\AA$ is Cohen-Macaulay

\smallskip

\item 
$\C = \C_1 \otimes \C_2$,
$\C^{**} = \C_1^{**} \otimes \C_2^{**}$,
 $r(\AA) = r(\AA_1) \cdot r(\AA_2)$

\smallskip

\item If $(c_i)$ is a minimal reduction of $\C_i$, $i=1,2$, then $(c)= (c_1)\otimes (c_2)$ is a minimal reduction for $\C$ and
\[ \C/(c) = \C_1/(c_1) \otimes \C_2 \oplus C_1 \otimes \C_2/(c_2),\]
\[ \C^{**}/\C= \C_1^{**}/\C_1 \otimes \C_2 \oplus C_1 \otimes \C_2^{**}/\C_2,\]
\[ \cdeg(\AA) = \cdeg(\AA_1) \cdot \deg(\AA_2) + \deg(\AA_1) \cdot \cdeg(\AA_2),\]
\[ \ddeg(\AA) = \ddeg(\AA_1) \cdot \deg(\AA_2) + \deg(\AA_1) \cdot \ddeg(\AA_2).\]

\smallskip

\item   If $\AA_i$ is almost
Gorenstein, that is $\cdeg(\AA_i) = r(\AA_i) - 1$, then 
\[ \cdeg(\AA) = (r_1-1)\cdot \deg(\AA_2) + \deg(\AA_1) \cdot (r_2-1).\]
Suppose further, $\AA_1= \AA_2$, that is $\AA$ is the square. Then
\[ \cdeg(\AA) = 2 \deg(\AA_1) (r_1-1). \]
In this case, $\AA$ is almost
Gorenstein if $\cdeg(\AA) = r_1^2 - 1$, that is
\[ 2 \deg(\AA_1) = r_1 + 1.\]
\end{enumerate}

\end{Proposition}

\begin{Questions}{\rm
 \begin{enumerate}[(1)]
\item How to pass from $\ddeg(\AA)$ of a graded algebra $\AA$ to $\ddeg(\BB)$ of one of its Veronese subalgebras?

\smallskip

\item If $\SS$ is a finite injective extension of $\RR$, is $\ddeg(\SS) \leq [\SS:\RR]  \cdot \ddeg(\RR)$?
\end{enumerate}
}\end{Questions}

\subsubsection*{Hyperplane sections} \label{hpsection} A  desirable comparison is that 
  between $\ddeg(\RR)$ and 
$\ddeg(\RR/(x)$ for appropriate regular element 	$x$. [The so-called `Lefschetz type' assertion.] 
 We know that if $\C$ is a canonical module for $\RR$ then
$\C/x \C$ is a canonical module for $\RR/(x)$ with the same number of generators, so type is preserved under
specialization. However $\C/x \C$ may not be isomorphic to an ideal of $\RR/(x)$. Here is a case of good behavior. Suppose
$x$ is regular modulo $\C$. Then for the sequence
\[ 0 \rar \C \lar \RR \lar \RR/\C \rar 0,\]
we get the exact sequence
\[ 0 \rar \C/x\C  \lar \RR/(x) \lar \RR/(\C,x) \rar 0,\] so the canonical module
$\C/x\C$ embeds in $\RR/(x)$. We set $\SS = \RR/(x)$ and $\D = (\C, x)/(x)$ for the image of $\C$ in $\SS$.  We need to compare
$ \ddeg(\RR) = \deg(\RR/\C) - \deg(\RR/\C^{**})$
and $ \ddeg(\SS) = \deg(\SS/\D) - \deg(\SS/\D^{**})$.

\medskip

We don't know how to estimate the last term.
 We can choose $x$ so that $\deg(\RR/\C) = \deg(\SS/\D)$, but need, for instance to show that $\C^{**}$ maps into $\D^{**}$. Let $c$ 
 be as
 in Proposition~\ref{bidual} and pick $x$ so that $c,x$ is a regular sequence. Set $\C_1 = (c):\C$, $\C_2 = (c): \C_1$, $\D_1 = (c):\D$, $\D_2 = (c):\D_1$. We have
 $\C_1 \D \subset (c)\SS$ and thus 
  $\C_1\SS \subset \D_1$. This shows that $\D^{**} = (c):_{\SS} \D_1 \subset (c):_{\SS} \C_1$,  
   and thus 
$\D_1\C_2 \subset (e)$ but not enough to show
  $\D_2\subset \C_2\SS$.

\medskip

A model for what we want to  have is  
{ \cite[Proposition 6.12]{blue1} asserting that if $\C$ is equimultiple then
$ \cdeg(\RR) \leq \cdeg(\RR/(x))$.
For $\ddeg(\RR) $, in consistency with Conjecture~\ref{cdegvsfddeg}, we have
\begin{Conjecture} {\rm Under the  conditions above,
$ \ddeg(\RR) \geq \ddeg(\RR/(x))$.
}\end{Conjecture}

\begin{Proposition} Let $(\RR, \m)$ be a Cohen--Macaulay local ring of dimension $\geq 2$. If $x\in \m$ is a regular
element of $\RR$ and $\C$ is a canonical  ideal of $\RR$, then $\C/x\C$ is a canonical ideal of $\RR/x\RR$ if and only if $\height(x, \tr(\C) )\geq 2$.   
\end{Proposition}

\begin{proof}
If $\C$ is a canonical module of $\RR$, $\C/x\C$ is a canonical module of $\RR/x\RR$. We must show that it is
an ideal. For any minimal prime $\p$ of $x$ $\C_{\p}$ is an ideal which cannot contain  its trace by hypothesis by the height condition. It follows that $\C_{\p}$ is principal and thus the localization $\RR_{\p}$ is a Gorenstein
ring. This implies that modulo $x$ it is also Gorenstein, and thus its canonical modulo is isomorphic to
the ring. The converse is just a reversion of the steps.
\end{proof}

\section{Monomial subrings}
\noindent
Let $H$ be a finite subset of $\ZZ_{\geq 0}$, $\{ 0=a_0< a_1<\ldots< a_n\}$. For a field $k$ we denote by  $\RR = k[H]$ the subring of $k[t]$
  generated by the monomials $t^{a_i}$.  We assume that $\gcd(a_1, \ldots, a_n) = 1$. We also use the bracket notation
  $H= \left< a_1, a_2, \ldots, a_n\right>$ for the exponents.
For reference we shall use \cite[p. 553]{Eisenbudbook} and
\cite[Section 8.7]{VilaBook}.

\medskip

There are several integers playing roles in deriving properties of $\RR$, with the emphasis on those that lead to the determination of
$\ddeg(\RR)$.

\medskip

\begin{enumerate}[{\rm (1)}]

\item There is an integer $s$ such that $t^n\in \RR$ for all $n\geq s$. The smallest such $s$ is called the {\em conductor} of $H$ or of $\RR$, which
we denote by $c$ and $t^c k[t]$ is the largest ideal of $k[t]$ contained in $k[H]$.
 $c-1$ is called the Frobenius number of $\RR$ and reads its multiplicity, $c-1=\deg(\RR)$. 

\medskip

\item For any integer $a> 0$, say $a=c$ of $H$, the subring $\AA=k[t^a]$ provides for a Noether normalization of $\RR$. This permits the passage of many
properties $\AA$ to $\RR$, and vice-versa. $\RR$ is a free $\AA$-module and taking into account the natural graded structure
we can write
\[ \RR = \bigoplus_{j=1}^m \AA t^{\alpha_j}.\]
Note that $s = \sum_{j=1}^m  \alpha_j$.

\medskip

\item How to read other invariants of $\RR$ such as its canonical ideal $\C$ and $\red(\C)$ and its canonical degrees
$\cdeg(\RR)$ and $\ddeg(\RR)$?

\end{enumerate} 

\subsubsection*{Monomial curves}
Let
$\RR = k[t^a, t^b, t^c]$, $a<b< c, \ \gcd(a,b,c)=1$. Assume $\RR$ is not Gorenstein,  $\omega  = (1,t^s)$.

\medskip

It would be helpful to have available descriptions of the canonical ideal and attached invariants. 
Some of the  information  is available from the Frobenius equations:
\begin{eqnarray*}
x^{m_{11}} & - & y^{m_{12}}  z^{m_{13}}\\
y^{m_{21}} & - & x^{m_{22}}  z^{m_{23}}\\
z^{m_{31}} & - & x^{m_{32}}  y^{m_{33}}
 \end{eqnarray*}
 which can be expressed as the $2\times 2$ minors of the matrix (\cite{Herzog})
  \begin{eqnarray*}
\varphi = \left[ \begin{array}{ll}
x^{a_1} & z^{c_2} \\
z^{c_1} & y^{b_2} \\
y^{b_1} & x^{a_2}
\end{array}
\right]
\end{eqnarray*}

\subsubsection*{Calculating the canonical ideal and its bidual} Let $\RR= \AA/P$ be a Cohen-Macaulay integral domain and $\AA $ a Gorenstein local ring.
If $\codim P = g$, $\omega = \Ext^g_{\AA} (\RR, \AA )$ is a canonical module for $\RR$. Since $\RR $ is an integral domain, $\omega$ may
be identified to an ideal of $\RR$. 
A canonical module of $\RR$ is obtained by deleting one row and a column according to the following comments
 and mapping the remaining entries to $\RR$.

\medskip

\begin{enumerate}[{\rm (1)}]

\item Let $L= (x_1, \ldots,  x_g)$ be a regular sequence contained in $P$. Then 
\[ \omega =\Ext^g_{\AA}(\AA/P , \AA) = \Hom_{\AA/L} (\AA/P, \AA/L) = (L:P)/L.\]
 The simplest case to apply this formula is when $P = (L, f)$ so this becomes
take \[ \omega = (L:f)/L.\]
 
\item Suppose $P = I_2(\varphi)$ is a prime ideal of codimension two where
 \begin{eqnarray*}
\varphi = \left[ \begin{array}{ll}
a_1 & {c_2} \\
{c_1} & {b_2} \\
{b_1} & {a_2}
\end{array}
\right]
\end{eqnarray*}
 We are going to argue that if the $2 \times 2$ minors $x_1, x_2$,
 \begin{eqnarray*}
 x_1 & = & a_1 a_2 - b_1 c_2 \\
 x_2 & = & a_1 b_2 - c_1 c_2 
 \end{eqnarray*}
 that arise from keeping the first row form a regular sequence then
 for $f = b_1 b_2 - c_1a_2$,  we have
 \[ (x_1, x_2): f = (a_1, c_2).\]
With these choices we have
\[ (L: f)/L \simeq ({a_1}, {c_2},  P)/P.\]

Indeed the nonvanishing mapping 
\[ \omega = (L:f)/L \mapsto \RR/P\]
of modules of rank one over the domain $\RR/P$ must be an embedding.
\end{enumerate}

\medskip

\begin{Example}{\rm 
Let $\RR= k[t^a, t^b, t^c]$, $b-a=c-b=d$: Then $b=a+d, c=a+2d$ and
 \begin{eqnarray*}
\varphi = \left[ \begin{array}{ll}
x & y \\
z^{p} & x^{q} \\
y & z \\
\end{array}
\right],
\end{eqnarray*} Note $ p(a+2d)-qa= d$ from
$(p+1)(a+2d) = qa + (a+d)$  
and
 $(x,y) \mapsto
\C = (t^a,t^b)= t^a(1,t^d)$. By Proposition~\ref{bidual} we have 
 $(x):(x,y) =(x, y, z^p)$,
  \begin{eqnarray*} 
  (x):(x,y) &=&(x, y, z^p),\\
 \C^{**} &=& (x): (x,y,z^p) = (x, z, y) ,\\
    \ddeg(\RR) &=& \lambda (\RR/\C) - \lambda(\RR/\C^{**}) =   \lambda(\RR/(x,y)) - \lambda(\RR/(x,y,z)) = p-1.   
\end{eqnarray*}
}\end{Example}

\begin{Proposition} \label{preaGor}
Let
$\RR= k[t^a, t^b, t^c]$ be a monomial ring
and denote by $\varphi$ its Herzog matrix
\begin{eqnarray*}
\varphi = \left[ \begin{array}{ll}
x^{a_1} & z^{c_2} \\
z^{c_1} & y^{b_2} \\
y^{b_1} & x^{a_2}
\end{array}
\right].
\end{eqnarray*}
 Then
\[    \ddeg(\RR)=  {a_1} \cdot  {b_2}   \cdot   {c_2}. \]
\end{Proposition}

\begin{proof}
 From $\varphi$  we 
take
  $\C = (x^{a_1}, z^{c_2})$, where harmlessly we chose $a_1\leq a_2$.
   Then  
  \begin{eqnarray*} 
  (x^{a_1}): (x^{a_1}, z^{c_2}) &=&  (x^{a_1}, z^{c_1}, y^{b_1} ),\\
 \C^{**} &=& (x^{a_1}): (x^{a_1},y^{b_1},z^{c_1}) = (x^{a_1}, y^{b_2}, z^{c_2}), \\
    \ddeg(\RR) &=& \lambda (\RR/\C) - \lambda(\RR/\C^{**}) =   \lambda(\RR/(x^{a_1},z^{c_2})) - \lambda(\RR/(x^{a_1},y^{b_1},z^{c_2})) .   
\end{eqnarray*}
and since
\[ (x^{a_1}, z^{c_2}) =(x^{a_1}, y^{b_1} y^{b_2},   z^{c_2})\]
we have
\begin{eqnarray*}
\ddeg(\RR) &=&  {a_1} \cdot ({b_1} + b_2) \cdot   {c_2}
- 
   {a_1} \cdot  {b_1}   \cdot   {c_2}\\
  & = &
     {a_1} \cdot  {b_2}   \cdot   {c_2}. 
 \end{eqnarray*}
\end{proof}

\begin{Remark}
{\rm 
For the monomial algebra $\RR=k[t^a, t^b, t^c] $ 
The value of $\ddeg (\RR) $ is also
calculated in
  \cite[Proposition 7.9]{Herzog16}.
  According to \cite[Theorem 4.1]{GMP11},  $\cdeg(\RR) = a_2 \cdot b_2\cdot c_2$, which supports 
   the  Comparison Conjecture~\ref{cdegvsfddeg}.
 
}\end{Remark}

\section{Rees algebras} 
\noindent
Let $\RR$ be Cohen-Macaulay local ring and $I$ an ideal such that the Rees algebra $S = \RR[I t]$ is Cohen-Macaulay.  We consider
a few classes of such algebras. We denote by $\C$ a canonical ideal of $\RR$ and set $G = \gr_I(\RR)$. 

\medskip

\subsubsection*{Rees algebras with expected canonical modules}  (See \cite{HSV87} for details.) This means $\omega_S = \omega_{\RR}(1,t)^m$, for some
$m \geq -1$. This will be the case when $G = \gr_I(\RR)$ is Gorenstein (\cite[Theorem 2.4,  Corollary 2.5]{HSV87}).
 We actually assume $I$ of codimension at least $2$.
We first consider the case $\C = \RR$. Set $\D = (1,t)^m$, pick  $a$ is a regular element in $I$ and its initial form $\overline{a}$ is regular on $G$,
finally  replace $\D$ by $a^m\D\subset S$.

\begin{Proposition}
Let $\RR$ be a Gorenstein local ring, $I$ an ideal of codimension at least two and $S$ its Rees algebra. If the canonical module of $S$ has the expected
form then $S$ is Gorenstein in codimension less than $\codim (I)$, in particular $\ddeg(S) = 0$. 
\end{Proposition}

\begin{proof}
  It is a calculation in \cite[p. 294]{HSV87} that $ S: (1,t)^m = I^m S $. It follows that
\[ (I^m, I^m t) \subset I^m S \cdot (1,t)^m. \]
Since $\codim (I^m, (It)^m) = \codim (I, It) = \codim (I) + 1$, the assertion follows. Note that this implies 
 that $\omega_S$ is free in codimension one and therefore it is reflexive by a standard argument.
 Finally by Proposition~\ref{genddeg1}, $\ddeg( S) = 0$, and therefore   that $\cdeg(S)=0$.
\end{proof}

\begin{Remark}{\rm
How about the general case, when $\RR$ is not Gorenstein?
 Let 
us see whether we can argue these points directly.

\begin{enumerate}[{\rm (1)}]

\item $a^mS: I^m S = (a, at)^m$:
$S\cap  a^m S : S I^{m+1} \cap $ 

\[ a^m I S : I^{m+1}S = (a^m I + a^m I^2 t + \cdots + a^m I^{i+1} + \cdots): _S I^m = I^{m}S.
 \]
 
\item Let us calculate
$ \ddeg(S)=  \deg(S /\D) -
\deg(S/\D^{**})$:  
 \[\deg(S/I^{m-1}S) = 
\sum_{i=0}^{i=m-1}
 \deg(I^i S/I^{i+1}S) = m \deg(G).
  \]

\item
According to \cite{HSV87}, $a^i(1,t)^i$ is a Cohen-Macaulay ideal for $0\leq i\leq m$. We also have that $\overline{S } =S/a(1,t) = \overline{\RR}[\overline{I}t]$,
where $\overline{\RR} = \RR/a\RR$, $\overline{I} = I/a\RR$.

\smallskip

\item The ideal $J = a(1,t)$ is free in codimension $1$:
$(a,at)^m \cdot (a,at)^{-m} = (1,t)^{m} \cdot I^m S$ is an ideal of codimension $2$. Therefore for all $i\geq 0$, $\deg(J^i/J^{i+1}) = \deg(S/J)$.

\smallskip

\item From
$ \deg(S/a^m(1,t)^m  )$: 
\[\deg(S/a(1,t)) + \deg(a(1,t)/a^2(1,t)^2) + \cdots + \deg(a^{m-1}(1,t)^{m-1}/a^{m}(1,t)^m).\]

\item We get $\deg(S/\D)  =m \deg(\overline{S})$. 

\end{enumerate}}
\end{Remark}

\medskip

\subsubsection*{Rees  algebras of minimal multiplicity} Let $(\RR, \m)$ be a  Cohen-Macaulay local ring, $I$ an $\m$-primary ideal,  and $J$ a minimal reduction of $I$ with $I^2 = JI$. 
From the exact sequence
 \[  0\rar I\RR[Jt] \rar \RR[J] \rar \RR[Jt]/I\RR[Jt] = \gr_J(\RR) \otimes_{\RR} \RR/I 
 \rar 0,\] 
  with the middle and right terms being Cohen-Macaulay, we have that $I\RR[Jt] = I \RR[It]$ is  Cohen-Macaulay. 
With $\gr_I(\RR)$ Cohen-Macaulay and $I$ an ideal of reduction number at most $1$,  
we have  if $\dim \RR > 1$ then $\RR[I t]$ is Cohen-Macaulay.
A source of these ideals arises from irreducible ideals in Gorenstein rings such as  $I= J:\m$, where  $J$ is a parameter ideal.

\medskip

Let $S_0 = \RR[Jt]$ and $S = \RR + I t S_0$.
If $\RR$ is Gorenstein, then the canonical module of $S_{0}$ is  $\omega_{S_0} = (1,t)^{d-2}S_0$ and by the change of rings formula, we have
\[ \omega_{S} = \Hom_{S_0}( S, (1,t)^{d-2}S_0)= (1,t)^{d-2}S_0: S.\]

If $d=2$, then the canonical module $\omega_S$ is the conductor of $S$ relative to $S_0$. In the case
$ \m S_0$ is a prime ideal of $S_0$, so the conductor could not be larger as it would have grade at least two and then
$S = S_0$:  $\omega_S = \m S_0$.
If $d > 2$, then  $\m (1,t)^{d-2}S_0$ will
 work.

\medskip

Note that $S_0$ is Gorenstein in codimension $1$. If $P$ is such prime and   $P \cap \RR=\q\neq \m$,
then $(S_0)_\q = S_\q$, so $S_P=(S_0)_P$ is Gorenstein in codimension $1$. Thus we may assume
  $P\cap \RR = \m $, so that $P = \m S_0$.
   For $S_P$ to be Gorenstein  would mean $\Hom(S, S_0) \simeq S$ at $P$, that is ${\ds (S_0:  S)_P }$ is a principal 
   ideal of $S_P$ (see next Example). From \cite[Theorem 3.3.7]{BH},
 $\D = \Hom(S, S_0) = \m S_0=\m S$.   Thus $\deg(S/\D) = 2$ and since $\D^{**} \neq \D$, $\deg(S/\D^{**}) = 1$.
It follows that $\ddeg(S) = 1$.

\begin{Example}{\rm  Let $\RR = k[x,y]$,  and $I = (x^3, x^2 y^2, y^3)$. Then for the reduction $Q = (x^3,y^3)$, we have $I^2 = QI$. The Rees
algebra $S = \RR[It]= k[x,y,u,v,w]/L$, where 
$L=(x^2 u-xv, y^2w-y v, v^2-xyuw)$ is given by the $2\times 2$ minors of
${\ds 
\varphi = \left[ \begin{array}{ll}
v & xw \\
yu & v \\
x & y
\end{array}
\right],
}$
whose content is $(x,y,v)$. It follows that $S$ is not Gorenstein in codimension $1$ as it would require a content of codimension at least four.  
Indeed, setting $\AA = k[x,y, u, v, w]$, a projective resolution of $S$ over $\AA$ is defined by
the mapping $\varphi:  \AA^2 \rar \AA^3$ which dualizing gives
\[ \C=\Ext^2_{\AA}(\AA/L, \AA) = \Ext^1_{\AA}(L, \AA)= \coker(\varphi^{*}).\]
It follows that $\C$ is minimally generated by two elements
at the localizations of $\AA$ that contain $(x,y,v)$.
}\end{Example}

\medskip

\subsubsection*{Rees algebras of ideals with the expected defining relations} Let $I$ be an ideal of the Cohen-Macaulay local ring $(\RR,\m)$ [or a polynomial ring $\RR= k[t_1, \ldots, t_d]$ over the field 
$k$] 
with a presentation
\[ \RR^m \stackrel{\varphi}{\lar} \RR^n \lar I =(b_1, \ldots, b_n)\rar 0.\]  Assume that $\codim (I) \geq 1$ and that the entries of $\varphi$ lie in $\m$.  
 Denote by $\LL
$ its ideal of relations
\begin{eqnarray*}
 0 \rar \LL \lar \SS = \RR[\TT_1, \ldots, \TT_n] \stackrel{\psi}{\lar}
\RR[It] \rar 0,  \quad \TT_i \mapsto b_it .
\end{eqnarray*} $\LL$ is a graded ideal of $\SS$ and its component in degree $1$ is generated by the $m$ linear forms
\[ \ff= [\ff_1, \ldots, \ff_m] = [\TT_1, \ldots , \TT_n] \cdot \varphi.\]

\subsubsection*{Sylvester forms}\index{Sylvester form}
 Let
$\mathbf{f}= \{\ff_1, \ldots, \ff_s\} $ be a set of polynomials in
 $\LL\subset \SS=\RR[\TT_1, \ldots , \TT_n]$ and let $\mathbf{a}= \{a_1, \ldots,
a_s\}\subset \RR $. If  $\ff_i \in (\mathbf{a}) \SS$ for all $i$, we can
write
\[ \mathbf{f}= [\ff_1, \ldots, \ff_s] = [a_1, \ldots, a_q] \cdot\AA =
\mathbf{a}\cdot \AA,\] where $\AA $ is an $s\times q$ matrix with entries
in $\SS$. We call $(\aa)$ a $\RR$-content of $\mathbf{f}$.
 Since  $\aa\not \subset L$, then the $s \times s$ minors of $\AA$ lie in $\LL$.
   By an abuse of terminology\index{Sylvester form}, we refer 
   to such determinants 
    as
{\em
Sylvester forms}, or the {\em Jacobian
duals}\index{Jacobian dual},  of $\mathbf{f}$ relative to $\mathbf{a}$.
If
 $\aa = I_1(\varphi)$, we write $\AA = \BB(\varphi)$, and call it the Jacobian dual of  $\varphi$. Note that if $\varphi$ is a matrix with linear entries
 in the variables $x_1, \ldots, x_d$, then
 $\BB(\varphi)$ is a matrix with linear entries in the variables $\TT_1, \ldots, \TT_n$.

\begin{Definition}\label{expectrel} {\rm Let $I = (b_1, \ldots, b_n)$ be an ideal with a presentation as above and let = $\aa  (a_1, \ldots, a_s) = I_1(\varphi)$. The Rees algebra
$\RR[It]$ has the {\em expected relations} if
\[ \LL =  ( \TT \cdot \varphi, I_s(\BB(\varphi))).\]
}\end{Definition}

There will be numerous restrictions to ensure that $\RR[It]$ is a Cohen-Macaulay ring and that it is amenable to the determination of its canonical degrees. We consider
some special cases grounded on \cite[Theorem 1.3]{MU} and \cite[Theorem 2.7]{UlrichU96}

\begin{Theorem}\label{rees1}Let $\RR = k[x_1, \ldots,   x_d]$ be a polynomial ring over an infinite field,
let $I$ be a perfect $\RR$-ideal of grade $2$ with a linear presentation matrix $\varphi$, 
 assume that 
 $\nu(I) > d$ 
 and that $I$ satisfies  $G_d$ $($meaning $\nu(I_{\pp})\leq \dim \RR_{\pp}$ on the punctured spectrum$)$. 
  Then $ \ell(I) = d;$ $r(I) = \ell(I)-1$; $\RR[It]$ is Cohen-Macaulay, and $\LL = (\TT\cdot \varphi, I_d(\BB(\varphi))).$
 \end{Theorem}
    
 The canonical ideal of these rings is described in \cite[Theorem 2.7]{UlrichU96} (for $g=2$)
 
 \begin{Proposition} 
 Let $I$ be an ideal of codimension two satisfying the assumptions of Theorem~\ref{rees1}.     
 Then $\ddeg(\RR[It]) \neq 0$
and similarly $\cdeg(\RR[It])\neq 0$. 
 \end{Proposition}
 
 \begin{proof}
Let $J$ be a minimal reduction of $I$ and set $K = J:I$. Then $ \C =K t  \RR[It]$ is a canonical ideal of $\RR[It]$ by \cite[Theorem 2.7]{UlrichU96}.
Let $a$ be a regular element of $K$. Then
 \[ (a): \C = \sum L_i t^i  \subset \RR[It], \quad \mbox{and} \quad bt^i \in L_i \;\;  \mbox{if and only if} \;\; b  KI^j \subset a I^{i+j}.\]
Thus for $b\in L_0$, $b\cdot K \subset (a)$ so $b= r\cdot a$ since $\codim (K) \geq 2$. Hence $L_0 = (a)$.
 For $b t\in L_1$, $b K \subset aI$, $b= ra$ with $r \in I:K$. As $rKI \subset I^2$, we have $L_1 = a(I:K)t$.
 For $i \geq 2$, $b\in L_i$, we have $b = rat^i$ with $r K \subset I^i$. Hence $L_i =a (I^i:K)t^i$.  In
 general we have $L_i = a(I^{i}: K)t^i$. Therefore, we have
 \[ \begin{array}{rcl}
 {\ds (a): \C} &=& {\ds  a(\RR+ (I:K) t + (I^2:K)t^2 +  \cdots ),} \vspace{0.1 in} \\
{\ds \C^{**} } &=& {\ds   (a) : ((a): \C) =  \sum_j L_jt^j = \RR[It] :   ( \RR + (I:K) t + (I^2:K)t^2 +  \cdots  ) }.
\end{array} \]
For $b\in L_0$ and $i\geq 1$,  $b(I^i:K) \subset I^i$ and thus $b\in \bigcap_i I^i : (I^i:K)=L_0$.  In general it is clear that $KI^i\subset L_{i}$. Note that
$L_0 = \RR: (\RR:K) = \RR$. It follows that $\C \neq \C^{**}$ and hence $\ddeg(\RR[It]) \neq 0$.
Recall that the vanishing of either of the functions $\cdeg$ or $\ddeg$ holds if and only if the ring is Gorenstein in codimension $1$. Therefore $\cdeg(\RR[It]) \neq 0$.
 \end{proof}

\section{Canonical degrees of 
 $\AA=\m:\m$}
\noindent
Let $(\RR, \m)$ be a CM local ring of dimension 1,  and set $\AA= \Hom(\m,\m)= \m:_Q\m$.  Assume that $\RR$ is not a DVR. 
Let $\C$ be its canonical ideal.
 (We can also discuss some of the same questions by replacing $\m$ by a prime ideal $\p$ such that
$\RR_{\p}$ is not a DVR.)

\subsubsection*{General properties}

We begin by collecting elementary data on $\AA$.
\medskip

\begin{enumerate}[{\rm (1)}]

\item $\AA = \m:\m \subset \RR:\m$ and since $\m \cdot (\RR:\m) \not = \RR$ [as otherwise $\m$ is principal] $\RR:\m = \AA$. 
\medskip

\item If $\RR$ is a Cohen-Macaulay local ring of dimension one that is
not a DVR, then $\AA = \RR:\m$  as $\m\cdot \Hom(\m,\RR) \subset \m$ and therefore
\[ \AA = \m:_Q\m = \RR:_{Q} \m= 1/x \cdot ((x):_{\RR}\m).\] Indeed,
if $q\in \m:_Q \m$ and $x$ is a regular element of $\m$, let $a=qx\in \m$. Then $q = (1/x)a$ and \[a\m = qx\m =
x q\m \subset x\m.\]
Thus $\AA = (1/x)(	(x):\m )$, which
  makes $\AA$ amenable by calculation using  software such as {\em Macaulay2} (\cite{Macaulay2}).
\medskip

\item A relevant point is to know when $\AA$ is a local ring. Let us  briefly consider some cases. 
 Let $L$ be an ideal of the local ring $\RR$
and suppose
 $L = I \oplus J$, $L = I + J$, $I\cap J = 0$, is a non-trivial decomposition.
  Then $I \cdot (J: I) = 0$ and $J\cdot (J:I) = 0$ and thus if 
 $\grade(I,J) = 1$ [maximum possible by the Abhyankar--Hartshorne Lemma], then $I:J=0$. It follows
 that 
 \[ \Hom(L,L) = \Hom(I,I) \times \Hom(J,J),\] and therefore $\Hom(L,L)$ is not a local ring.
 
 \medskip

  \item   Suppose $\RR$ is complete, or at least Henselian. If $\AA$ is not a local ring, by
  the Krull-Schmidt Theorem $\AA$ admit a non-trivial decomposition of $\RR$-algebras
  \[ \AA = \BB \times \CC.\]
 Since $\m = \m \AA$, we have a decomposition $\m = \m\BB \oplus \m \CC$. If we preclude  such decompositions then
 $\AA$ is a local ring. Among these cases are: analytically irreducible rings, in particular they include the localization 
 of any monomial ring. 
\end{enumerate}

\subsubsection*{Number of generators of $\AA$}

\begin{enumerate}[{\rm (1)}]

\item Since $\m$ is an ideal of both $\RR$ and $\AA$, $\RR:\AA = \m$. Thus $\AA/\RR \simeq (\RR/\m)^n = k^n$.
The exact sequence
\[ 0\rar \RR \lar \AA \lar k^n \rar 0,\]
yields
\[ 0\rar \RR/\m \lar \AA/\m=\AA/\m\AA \lar k^n \rar 0,\]
gives $\nu(\AA) = n + 1$ since $\RR \not \subset \m \AA = \m$.

\medskip

\item $\D = \C:\AA$ is the canonical ideal of $\AA$, according to \cite[Theorem 3.3.7]{BH}.
 Applying $\Hom(\cdot , \C)$ to the exact sequence above
we get
\[ 0\rar \D \lar \C \lar \Ext^1(k^n, \C) \simeq k^n \rar 0.\] 
Thus $\C/\D \simeq k^n$ and $\m \C \subset \D $. As
\[0 \rar \m\C \lar \C \lar  k^r\rar 0 \]  
\[ \D/\m \C \simeq k^{r-n}.\]
\end{enumerate}

 \begin{Theorem}\label{NGens} Let $(\RR, \m)$ be a Cohen-Macaulay local ring of dimension one and type $r(\RR)=r$.
 Then \[\AA = \Hom(\m,\m) = \m:_ Q \m= x^{-1}\cdot ((x):\RR)\] 
 for any regular element $x\in \m$. Furthemore
 \[ \boxed {
 \nu_{\RR}(\AA) = r+1.}\]
 \end{Theorem}

\begin{proof}
Since $\AA=x^{-1}((x):\m)$ for any regular element of $\m$, we have

  $\AA = 1/x\cdot ((x):\m)$, which gives
$ \nu(\AA) =\lambda((x):\m)$ and since
\[ ((x):\m)/x\RR \simeq k^r\] then
\[r\leq \nu((x):\m) \leq r+1.\]

\smallskip

Writing $(x):\m = (x, y_1, \ldots, y_r)$,  $n=r-1$ would mean that $x$ is not a minimal generator of $(x):\m$, so 
\[ x\in \m ((x):_{\RR} \m).\]
In particular $x\in \m^2$.
If you preclude
this with an initial
 choice of $x\in \m\setminus \m^2 $, we would have $r=n$ always.  
 
\end{proof}

\begin{Corollary} If $(\RR,\m)$ is a local ring and $\C$ is the 
canonical ideal, resp. module,  of $\RR$ then $\D = \m \C$ is the canonical ideal, resp. module, of $\AA$.
\end{Corollary}

\subsubsection*{Canonical invariants of $\AA$}

The driving questions are what are
 $r(\AA)$, $\cdeg{\AA}$ and $\ddeg(\AA)$ in relation to
 the invariants of $\RR$. Our main calculation is the following result:

\begin{Theorem} \label{TCdeg} Suppose $(\AA,M)$  is a local ring. Then 
\[  \boxed{ \cdeg(\AA) = e^{-1}[ \cdeg(\RR) + \rme_0(\m) - 2r].} \]
\end{Theorem}

\begin{proof}
 The equality $n=r$ means that $\D = \m \C$. In particular if $(c)$ is a minimal reduction of $\C$ and $
(x)$	is a minimal reduction of $\m$ then $(cx)$ is a minimal reduction of $\D$. This gives 
that
if $[e = \AA/M:\RR/\m]$ is
the   relative degree then
\begin{eqnarray*}
 \rme_0(\D, \AA)& =& e^{-1}\cdot  \lambda_{\RR}(\AA/xc\AA) = e^{-1}
 [  \lambda_{\RR}(\AA/x\AA)+ \lambda_{\RR}(x\AA/xc\AA)] \\ 
  &= &  e^{-1} \cdot[ \lambda_{\RR}(\RR/ x\RR) + \lambda_{\RR}(\RR/c\RR)] \\
 & =&
  e^{-1} [   \rme_0(\C) + \rme_0(\m)].
  \end{eqnarray*}
 On the other hand
 \begin{eqnarray*}
 \lambda_{\AA}(\AA/\D) & =& e^{-1}\cdot \lambda_{\RR}(\AA/\m \C) = e^{-1}\cdot [\lambda_{\RR}(\AA/\RR) 
 + \lambda_{\RR}(\RR/\C) + \lambda_{\RR}(\C/\m \C)] \\
  &= & 
  e^{-1}\cdot
   [2r + \lambda_{\RR}(\RR/\C)].
\end{eqnarray*}
These equalities  give
\[ \cdeg(\AA) = e^{-1}\cdot [ \rme_0(\C) -  \lambda_{\RR}(\RR/\C) + \rme_0(\m) - 2r] = 
e^{-1}[
 \cdeg(\RR) + \rme_0(\m) - 2r],\]
  as desired.
 \end{proof}

 \begin{Corollary} 
 $\AA$ is a Gorenstein ring if and only if
 \[ \cdeg(\RR) + \rme_0(\m) - 2r=0.\]
   In particular $\m \C = \AA z$, so $\m$ and $\C \AA$ are principal ideals of $\AA$.
\end{Corollary}

\medskip

\begin{Remark}{\rm If $\m$ is a maximal ideal but $\AA$ is semilocal with maximal ideals $\{M_1, \ldots, M_s\}$, 
we can still obtain a formula for $\cdeg(\AA)$ as a summation of the $\cdeg(\AA_{M_i})$, as 
\[ \cdeg(\AA) =  \sum_i \cdeg (\AA_{M_i}) =
[
 \cdeg(\RR) + \rme_0(\m) - 2r] \cdot (\sum_i e_i^{-1}) .\]
}\end{Remark}

 \begin{Example}{\rm

Let $(\RR,\m)$ be a Stanley-Reisner ring of dimension one. If $\RR=k[x_1, \ldots, x_n]/L$, $L$ is generated by all the binomials
 $x_ix_j$, $i\neq j$. Note that $\m = (x_1)\oplus \cdots \oplus (x_n)$ and since for $i\neq j$ the annihilator
 of $\Hom((x_i), (x_j))$ has grade positive,  the observations above show that  
 \[ \AA =\Hom(\m,\m) = \Hom((x_1), (x_1))\times \cdots \times \Hom((x_n),(x_n))= k[x_1] \times \cdots \times k[x_n].\]

To determine $\cdeg(\RR)$, we already have that $\rme_0(\m)=n$ so let us calculate $r=r(\RR)$.
The Hilbert series of $\RR$ is easily seen to be
\[ {{1 + (n-1)t}\over {1-t}},\]
so $r = n-1$, which yields
\[ \cdeg(\RR) = 2r -\rme_0(\m) = 2(n-1)-n = n-2=r-1.\]
Thus $\RR$ is almost Gorenstein (\cite[Theorem 5.2]{blue1}). This also follows from
 \cite[Theorem 5.1]{GMP11}.

\medskip

Some of this argument applies in higher dimension. For instance:

\medskip

If $\RR = k[x_1, \ldots, x_n]/L$ is an unmixed Stanley-Reiner ring, $L=\bigcap_{1\leq i\leq s} P_i$ is an irreducible
representation then the natural mapping
\[ \RR \mapsto \prod_{1\leq i\leq s} k[x_1, \cdots , x_n]/P_i\] 
gives an embedding of $\RR$ into its integral closure. Note that each factor is a polynomial ring of the same dimension as $\RR$.

}\end{Example}

\medskip

Let us determine other invariants of $\AA$:

\begin{Corollary} For $\AA$ as above

\begin{enumerate}[{\rm(1)}]
\item The type of $\AA$ is
 \[r(\AA) = \nu_{\AA} (\D) = e^{-1} \nu_{\RR}(\D) =e^{-1}\nu_{\RR}(\m \C) = 
 e^{-1} \lambda(\m\C /\m^2\C).\]

\item For $\AA$ to be AGL requires
\[ \cdeg(\AA) = r(\AA)-1 = e^{-1} \lambda(\m\C /\m^2\C) -1 = e^{-1}[\cdeg(\RR)   +   \rme_0(\m)
  -  2r] .
 \]
  
\end{enumerate}
\end{Corollary}

\subsubsection*{Additional invariants}
\begin{enumerate}[{\rm(1)}]

\item Next we would like to calculate under the same conditions
\[\ddeg(\AA)= \lambda_{\AA}(\AA/\D) - \lambda_{\AA}(\AA/\D^{**}).\]
The first term is as above
\[ \lambda_{\AA}(\AA/\D) 
= 1/e\cdot [2 r+ \lambda(\RR/ \C)].\]

\item We  need
 $\D^{**}$ or $\D \cdot \D^{*}$. We must pick an element $x\in \D$.

\medskip
\item Since $\AA = x^{-1}((x):\m)$ and $\D = \m \C\AA$
\[ \AA:\m = ? \]
\[ \D^{*} = [\AA:\m]:\C = \]
\[ \D^{*}\cdot \D = [[\AA:\m]:\C]\cdot \C \m = \]
\[ \D^{*} = \AA:\D = x^{-1}((x):\m):\m \C= x^{-1}[(x):\m):\m]:\C\]

\medskip

\item Would like at least to argue that
\[ \cdeg(\RR) \geq \ddeg(\RR) \Longrightarrow \cdeg(\AA) \geq \ddeg(\AA).\]
\end{enumerate}
 
\begin{Example}
{\rm We return to the ring
  $\RR = \mathbb{Q}[t^5, t^7, t^9], 
 \m = (x,y, z) $. We have a presentation $\RR = \mathbb{Q}[x,y,z]/P$, with $P = (y^2-xz, x^5-yz^2, z^3-x^4 y)$. Let
  us examine some properties of $\RR$ and $\AA = \m:\m$.
\begin{enumerate}[{\rm(1)}]
\item The canonical module is
 $\C = (x,y)$, and a minimal reduction $(c)=(x)$.   It gives $\red(\C) = 4$.
  
\item  $(c):\C \neq \m$ so $\RR$ is not almost Gorenstein. However $\C^{**} = (c):  [(c):\C]$ satisfies  [by another {\em Macaulay2} calculation]
$\lambda(\C^{**}/\C) = 1$, so $\C^{**} = L$. This shows that 
\[
\ddeg(\RR) =\lambda(C^{**}/C) 
= 1\quad \mbox{\rm and} \quad \cdeg(\RR) = \lambda(\C/(c))  = 2.\]

\item Since $\rme_0(\m) = 5$,
 $r=2$ and $e=1$, we have
\[ \cdeg(\AA)= 2-4+5 = 3.\] 
\item $\red(\D) = \red(\m \C) = 2$. Note that (x) is a minimal reduction of both $\C$ and $\m$.

\item To calculate $\D^{**}$ we change $\C$ to $x \C$. We then get $\lambda(\AA/\D) = 11$ and $\lambda(\AA/\D^{**})= 9$, 
and so   $\ddeg(\AA) = \lambda(\AA/\D)-\lambda(\AA/\D^{**}) = 2$.
\end{enumerate}
}\end{Example}

\section{Effective computation of bi-canonical degrees}

\noindent
 A basic question is how to calculate canonical degrees. We will propose an approach that applies to the bi--canonical 
 degree and suggests an extension to rings when the canonical module 
  $\C$ is not an ideal.

 \begin{Definition}
 {\rm Let
  $\RR$ be a Cohen-Macaulay local ring and $\C$ one  canonical module.
  Consider the natural mapping
 \[ 0\rar E_0 \lar
 \C \lar \C^{**} \lar E_1 \rar 0,\]
  define \[ \boxed{\ddeg_{\C}(\RR) = \deg(E_1) + \deg(E_0).}\] 
}\end{Definition} 

The primary setting  of our calculations is the following construction of Auslander \cite{AusBr}.

\begin{Definition} \label{ausdualdef} {\rm Let $E$ be a finitely generated $\RR$-module with a
projective presentation
\[ F_1 \stackrel{\varphi}{\lar} F_0 \lar E \rar 0.\]
The {\em Auslander dual} of $E$ is the module $D(E)=
\coker(\varphi^*)$,
\begin{eqnarray} \label{ausdual}
0\rar E^*\lar  F_0^* \stackrel{\varphi^*}{\lar} F_1^* \lar D(E) \rar 0.
\end{eqnarray}
}\end{Definition}

The module
$D(E)$ depends on the chosen presentation but it is unique up to
projective summands. In particular the values of the functors
$\Ext_{\RR}^i(D(E),\cdot)$ and $\Tor_i^{\RR}(D(E), \cdot)$, for $i\geq 1$,
are independent of the presentation. Its use here lies in the
following result (see \cite[Proposition 2.6]{AusBr}):

\begin{Proposition}\label{Adual}
 Let $\RR$ be a Noetherian ring and
 $E$  a finitely generated $\RR$-module. There are two exact
sequences of functors:
\begin{eqnarray} \label{adual1}
\hspace{.4in} 0 \rar \Ext_{\RR}^1(D(E),\cdot) \lar E\otimes_{\RR}\cdot \lar
\Hom_{\RR}(E^*,\cdot) \lar \Ext_{\RR}^2(D(E),\cdot)\rar 0
\end{eqnarray}
\begin{eqnarray}  \label{adual2}
\hspace{.4in}   0 \rar \Tor_2^{\RR}(D(E),\cdot) \lar E^*\otimes_{\RR}\cdot \lar
\Hom_{\RR}(E,\cdot) \lar \Tor_1^{\RR}(D(E),\cdot)\rar 0.
\end{eqnarray}
\end{Proposition}

    \begin{Corollary} 
   Let $\RR$ be a Cohen-Macaulay local ring and $\C$ one of its canonical modules. 
     If $D(\C)$ is the Auslander dual of $\C$,
\[ \boxed{E_0 =  \Ext_{\RR}^1(D(\C), \RR) , \quad  E_1 =  \Ext_{\RR}^2(D(\C), \RR).} \]
\end{Corollary}
Now we give a presentation of $D(\C)$ in an important case suitable for computation.

\begin{Remark}
{\rm Recall that to calculate $\Ext_{\RR}^i(M,N)$ we invoke the command $\Ext^i(f,g)$ where $f$ and $g$ are the
presentations of $M$ and $N$, respectively.
}
\end{Remark}

\begin{Theorem}
Let $\SS$ be a Gorenstein local ring 
 and $\RR =\SS/I $ a Cohen-Macaulay ring with a
   with minimal projective resolution   
\[ 0\rar \SS^p \stackrel{\varphi}{\lar} \SS^m \lar \cdots \lar \SS^n \stackrel{f}{\lar} \SS \lar \RR=\SS/I\rar 
0.\]Then
\begin{eqnarray*}
\C&=& \coker (\varphi^{*}\otimes \RR),\\
D(\C)&=& \coker (\varphi\otimes \RR),
\end{eqnarray*}
and therefore \[\boxed{ \ddeg (\RR) =
 \deg 
 (\Ext_{\RR}^1(\varphi\otimes\RR, \RR))
 +  \deg 
 (\Ext_{\RR}^2(\varphi\otimes\RR, \RR))}. \]
 Furthermore, $\C$ is an ideal if and only if $\height(I_{p-1}
 (\varphi)) \geq p+1$.
\end{Theorem}

\begin{Example}{\rm
 Consider the ring $\RR = \SS/I$, where $\SS$ is a Gorenstein local ring  
  and  $I $ is a Cohen-Macaulay ideal  given by a minimal resolution
\[ 0\rar \SS^{n-1}  \stackrel{\varphi}{\lar} \SS^n
\lar \SS \lar
\RR \rar 0.\] This gives 
\[ 0\rar  \SS \lar \SS^n \stackrel{\varphi^*}{\lar} \SS^{n-1} \lar \C \rar 0\]
where $\C$ is the canonical module of $\RR$.  We get the exact complexes 
\[  \RR^n \stackrel{\bar{{\varphi^*}}} {\lar} \RR^{n-1} \lar  \C \rar 0,\] 
\[ 0
\rar \C^*\lar   \RR^{n-1} \stackrel{\bar{\varphi}} {\lar} \RR^n \lar D( \C) \rar 0.\] 

\begin{enumerate}[{\rm (1)}]
\item In general
 we have that $\C$ is an ideal if and only $\codim \ (I_{n-2}(\varphi))\geq 3$. For instance
if 
\[ \varphi = \left[ \begin{array}{lll}
x & yz & xz \\
y^2 & zx & xy
\end{array}\right]
\]
\[ I = I_2(\varphi)=
(x^2z-y^3z, x^2y-xy^2z, x^2z^2-xy^2z)\]  
but $I_1(\varphi) $ has codimension $2$ so $\C$ is not an ideal.

\medskip 
\item Setting $Z = \ker(\varphi)$ we have the exact sequences
\[ 0 \rar \C^* \lar \RR^2 \lar \Hom(Z, \RR) \lar \Ext^1_{\SS}(\C, \RR)\rar 0\]   
\[ 0 \rar \Hom(Z, \RR) \lar \RR^3 \lar \RR \lar \Ext^{1}_{\SS}(Z, \RR) = \Ext^2_{\SS}(\C, \RR) \rar 0.\]
\end{enumerate}
}\end{Example}

For the homogeneous ring
defined by
\[ \varphi = \left[ \begin{array}{lll}
x^2 & yz & xz \\
y^2 & zx & xy
\end{array}\right],
\] we found
 $\deg(E_0) = 1$ and $\deg(E_1) = 6$, so $\ddeg(\RR)=7$.

\bigskip

{\bf Notes and Questions.}

\begin{enumerate}[{\rm (1)}]

\item  $D(\C)$ is fed into the Ext calculation to determine $E_0, E_1$. We only did homogeneous calculations.

\medskip

\item 
Would like some very explicit formulas. For instance suppose $I$ has type $2$ and $\C$ is an ideal

\[ 
0 \rar \SS^2  \stackrel{\varphi}
{\lar} \SS^3 \lar I \rar 0
.\]

Then $\height I_2(\varphi)=3$. We have $\C = \coker (\varphi^t)=
H_{n-2}(
K)$. This is to be completed and extended to all strongly Cohen-Macaulay ideals.

\medskip

\item In our calculations, we asked also for $\dim(E_0)$, $\dim(E_1)$ besides $\deg(E_0)$, $\deg(E_1)$. 

\medskip

\item We would like to verify that if $E_0\neq 0$ then $\dim E_0=\dim  E_1 $. 

\medskip

\item What are the properties of $\ddeg_{\C}(\RR)$? 

\medskip

\item If $E_0=0$, $E_1  = C^{**}/C$ and it is either $0$ or Cohen-Macaulay of dimension $d-1$.

\medskip

\item  If  $ E_0\neq 0$ it must have the same dimension as  $\C$ [which is a MCM], that is $d$. We guess that $E_1 $ has also dimension $ d$.      
If $\dim E_1<d$, localizing at one prime not in the support get a case with $E_0\neq 0$ but
$E_1 =0$.

\medskip

\item What if $\C$ is reflexive? 

\begin{itemize}

\item
Let  $\RR^m \lar \C^{*} \rar 0$, then we have
$0\rar \C^{**} \lar \RR^{m}$.  If $\C$ is reflexive it will then embed in $\RR^m$.

\medskip

\item Suppose $\RR$ has dimension $\leq 1$.
Localizing at a regular element we get
$\C_x$ splits off $Q^{m}$ since it is an injective $Q$-module (in fact its injective envelope). Thus $\C_x$  splits off a free module.

\medskip

\item From $\Hom(\C_x,\C_x) = Q$ we get that $\C_x=Q$.
 Thus $\C_x$ is an ideal of $Q$ and therefore $\C$ is an ideal of $\RR$.
 
 \medskip

\item If $\C$ is reflexive, in dimension one
by an earlier criterion $\RR$ is Gorenstein.

\end{itemize}
 
\end{enumerate}

\begin{Corollary} $E_0=0$ if and only if $\C$ is an ideal.
\end{Corollary}

Now must examine the vanishing of $E_1$. Consider
\[0  \rar H \lar \C \otimes \C^{*} \lar \mbox{\rm trace}(\C) =L \subset \RR\]

We have
\[ 0 \rar \Hom(L, \C) \lar \Hom(\C,\RR) \lar \Hom(H, \C) \rar 0.\]

\medskip

\section{Other generalizations}
\noindent
Let $\RR$ be a Noetherian local ring and $M$ a finitely generated MCM $\RR$-module. We would like to define a bicanonical degree
$\ddeg(M)$ with properties similar to $\ddeg(\RR)$.

\begin{enumerate}[(1)]
\item One of these is
 $\ddeg(M)=0$ if and only if $M$ is a Gorenstein module in codimension $1$. A Gorenstein module is a maximal Cohen-Macaulay module
of finite injective dimension.
 
 \medskip
 
 \item We should be careful here:  $\m$  is a bidual but not always principal. 
This shows the need for additional restrictions.

\medskip

\item
In Proposition~\ref{resddeg3} we give a class of rings and ideals with the required conditions.

 \medskip

\item Let $M$ be a MCM module and set $L = \Hom(M,\C)$. In analogy of the case above set
\[ 0 \rar E_0 \lar L \lar L^{**} \lar E_1 \rar 0,\] and define
\[ \ddeg(M) =\deg(E_0) + \deg(E_1).\]
If $\C$ is an ideal then  $E_0=0$. Here we would like to show that if $\ddeg(M) =0 $ then $M$ is Gorenstein in codimension one. This should mean $L$ is free. The usual approach is set $\AA = \Hom(L,L)$ and consider the natural mapping
 \[ \varphi:L\otimes L^* \lar \Hom(L,L) = \AA.\]
 The image $\tau$ of $\varphi$ is the ideal of $\AA $ of the endomorphisms of $L$ that factor thru projectives. 
 
 \end{enumerate}

 We would like to see whether a version of Proposition~\ref{resisddeg1} works here, that is
 
\begin{Proposition}\label{resisddegmod} Let $\RR$ be Cohen-Macaulay local ring of dimension $1$ with a canonical ideal $\C$. If $M$ is  torsionless, that is a submodule of a free module, 
 then
\[ \ddeg(M) = 
\lambda(\AA/\tau),\]
where $\tau$  is the trace of $M$.
\end{Proposition}

\begin{proof} 
We will show that
\[ \lambda(L^{**}/L) =
\lambda(\AA/\tau).\]
From the exact sequence
\[ 0
\rar K \lar L\otimes L^{*} \lar  \tau \rar 0\]
we have that $K$ has finite support and therefore
\[ \Hom(\tau,\C) = \Hom(L\otimes L^{*} ,\C) = \Hom(L, \Hom(L^{*}, \C)) = \Hom(L, \Hom(\Hom(L, \RR), \C))
\] From $\tr(\C) = \C\cdot
 \C^{*}$, we
 have 
\[  \Hom(\tr(\C),\C)=  \Hom(\C \cdot \C^{*},\C)= \Hom(\C^{*}, \Hom(\C,\C))  = \C^{**}.\]
 Now dualize the
  exact sequence
\[
0 \rar
\tr(\C)\lar  \RR \lar \RR/\tr(\C)\rar 0,\]   into $\C$
to obtain
\[0=\Hom(\RR/\tr(\C),\C)\rar  \C=\Hom(\RR,\C)\rar \C^{**}=\Hom(\tr(\C),\C)\rar  \Ext^1(\RR/\tr(\C),\C)\rar  0,\]
which 
 shows that $\C^{**}/\C$ and $\Ext^1(\RR/\tr(\C),\C)$ are isomorphic. Since by local duality $\Ext^1(\cdot,\C)$ is self-dualizing on modules of
 finite support, 
$\ddeg(\RR)= \lambda(\RR/\tr(\C)).$
\end{proof}

\section{Precanonical Ideals}
\medskip
\noindent
Let $I$ be an ideal of the Cohen-Macaulay local ring $\RR$. To attach a {\em degree}
to $I$ with properties that mimic $\ddeg(\RR)$ we would want it to satisfy three properties:
\begin{enumerate}[(1)]
\item $I $ is closed,	that is $\Hom(I,I)=\RR$.
\item
 If $I$ is principal then $\RR$
is Gorenstein, at least in codimension one.
\item
If $I$ is reflexive then $I$ is principal.
\end{enumerate}

\noindent This could be used to define: $\ddeg_I(\RR)= \deg(I^{**}/I)$. For  a class $\mathcal{I}$ of such ideals,
a question
is when this degree is independent of $I$, for non-principal ideals.
\begin{enumerate}[(1)]
\item Of course the issue is to know what the conditions above mean.
 A promising case to consider
 it that of semidualizing ideals/modules. These modules arose in several sources. In \cite{V74} a class of
 modules, termed {\em spherical} \index{spherical modules} were introduced by  conditions akin to canonical modules: the closedness 
 $\Hom_{\RR}(D, D ) = \RR$ and the rigidity conditions
 $ \Ext^i(D,D)=0 $, $i\geq 1$. 
 See \cite{CF, CSW1} . 
 
 \medskip
 
\noindent  We want to consider the question: If $D$ is such ideal and it is reflexive, is
 it  principal in codimension 1?  This would be a generalization of    
 \cite[Corollary 7.29]{HK2}. In \cite[p. 109]{V74} a presumed proof is not clear,
 but see \cite[Lemma 2.9]{TK}.  
 
 \medskip
 
 \item Another path is to consider is
 \[ \deg(I	/(c)),
\] where $c$ is some distinguished element of $I$.

 \end{enumerate}

\begin{Question}{\rm  
  Where to find semidualizing modules? How are the semidualizing modules of $\RR$, $\AA=\m\colon \m$ and 
 $\BB= \RR\ltimes \m$ related?
 }
 \end{Question}

Here is a teaser:

\begin{Proposition} \label{resddeg3}
 Let $\RR$ be a local ring of dimension one and finite integral closure,
and let $I$ be a closed ideal. If
$I$ is reflexive then $I$ is principal. 

\end{Proposition}

\begin{proof}
From $\tr(I) = L=I\cdot
 I^{*}$, we
 have 
\[  \Hom(\tr(I),\RR)=  \Hom(I \cdot I^{*},\RR)= \Hom(I\otimes I^{*},\RR)= \Hom(I, \Hom(I^{*},\RR) ) =
 \Hom(I,I)  = \RR,\] 
 and thus $L^{**} = \RR x$. If $x$ is a unit $L^{*}= \RR$ and therefore has height greater than one,
 which is not possible. Thus we have
 $ L = Mx$, where $M \subset \m$.
 Since $L^{**}=\RR x$ is integral over $L=Mx$  (\cite[Proposition 2.14]{CHKV}),
 we have for some positive integer $n$
 \[ \RR x^n = Mx\RR x^{n-1},\]
 and thus $\RR = M$, which is a contradiction.
\end{proof}

Better see \cite[Lemma 2.9]{TK}.

\subsubsection*{Almost
semidualizing ideals.}

\bigskip

\begin{Theorem}\label{Ulrich}
Let $(\RR,	\m)$ be a one-dimensional Cohen-Macaulay local ring that has a canonical ideal
and
 $S$ is an $\m$-primary ideal such that 
\begin{enumerate}[{\rm(1)}]
\item $\Hom(S,S)= \RR$, and
\item $\Ext^1(S,S)= 0$.
\end{enumerate}
If $c\in S$ 	is such that $S/(c)\simeq (\RR/\m)^n$, for some $n$, then
$S$ is a canonical ideal.

\end{Theorem}

\begin{proof}

 Consider the exact sequence of natural maps
\[ 
0\rar 
(c ) \lar S	\lar V \rar 0. \]
 Suppose $V = k^n$. Applying 
$\Hom(\cdot, S)$ we get
\[ 
0\rar \RR \lar c^{-1}S \lar \Ext^1(V,S) =\Ext^1(k, S)^n = \Hom(k, S/cS)^n \lar \Ext^1(S	,S) = 0.\]
Thus $ \Hom(k, S/cS)=k$.
Since $\RR/ S\simeq 
(c)/cS $ embeds in $S/cS$, the socle of $\RR/S$ is $\RR/\m$,
 so $S$ is an irreducible ideal.

The assertion would hold for  $xS$ for any regular element of $\RR$. 
		We now invoke Corolary~\ref{recogcor}.

\end{proof}

\noindent {\bf Comments:}
\begin{enumerate}[(1)]
\item  The ideals with these two conditions include the semidualizing ones. Theorem~\ref{Ulrich} asserts
that the almost Gorenstein rings of Goto (\cite{GMP11, GTT15}) can only occur when $S$ is a canonical ideal.

\medskip

\item What other interesting modules occur as $V=S/(s)$? We are going to assume that $S$  is Cohen-Macaulay of codimension one and $s$ is a regular element, in particular $V$ is a Cohen-Macaulay
module.
 	We will also assume that 
$S$ is closed. It always leads to $V 
 = \Ext^1(V, S)$.
 \medskip

\item Maybe {\bf closed} plus $\Ext^1(S,S)=0$ could be called strongly closed, or perhaps
1-closed... then those $\infty$-closed would
be the semidualizing ideals.

\medskip\item Suppose $S$ is 2-closed and $V = T^n$ with $T =\RR/\m^2$.

\end{enumerate}


\begin{thebibliography}{99}

\bibitem{Aoyama}{Y. Aoyama, Some basic results on canonical modules, {\em J. Math. Kyoto Univ.}
	{\bf 23} (1983), 85--94.}

	

\bibitem{AusBr}{M. Auslander and M. Bridger, {Stable module
theory}, {\em Mem. Amer. Math. Soc.} {\bf 94}, Providence, R.I.,
1969, 146pp.}

\bibitem{BF97}{V. Barucci and R. Fr\"oberg, One-dimensional almost Gorenstein rings, {\em J. Algebra} {\bf 188} (1997), 418--442.}

 
\bibitem{BV1}{J. P. Brennan and W. V. Vasconcelos, On the structure of closed ideals, {\em Math. Scand.} {\bf 88} (2001), 1--16.}



\bibitem{BrodSharp}{M. P.  Brodmann and S. Y. Sharp, {\em Local Cohomology}, Cambridge Studies in Advanced Mathematics {\bf 136},
Cambridge University Press, 1998, 2011}

\bibitem{BH}{W. Bruns and J. Herzog, {\em Cohen-Macaulay Rings},
 Cambridge University Press, 1993.}




\bibitem{CF}{L. W. Christensen, H. Foxby and H.Holm,
{Beyond totally reflexive modules and back: a survey on Gorenstein dimensions.} {\em Commutative algebra--Noetherian and non-Noetherian perspectives,} 101--143, Springer, New York, 2011. }


\bibitem{CSW1}{L. W. Christensen and S. Sather-Wagstaff, A Cohen-Macaulay algebra has only finitely
many semiduaizing modules, 	{\em Math. Proc. Camb. Phil. Soc.} {\bf 145}
(2008), 601--603.} 






\bibitem{CHV}{A. Corso, C. Huneke and W. V. Vasconcelos, On the integral closure of ideals, {\em Manuscripta Math.}
{\bf 95} (1998), 331--347.}


\bibitem{CHKV}{A. Corso, C. Huneke, D. Katz and W. V. Vasconcelos, Integral closure of ideals and annihilators of homology,
{\em Commutative Algebra},
Lecture Notes in Pure and Applied Mathematics {\bf 244} , 33--48, 
Chapman \& Hall/CRC, Boca Raton, FL, 2006.}



\bibitem{CP95}{A. Corso and C. Polini, Links of prime ideals and their Rees algebras, {\em J.  Algebra}
{\bf 178} (1995), 224--238.}





\bibitem{Ding}{S. Ding, A note on the index of Cohen-Macaulay local rings, {\em Comm. Algebra} {\bf 21} 
(1993), 53--71.}
  



\bibitem{ES}{P. Eakin and A. Sathaye, { Prestable ideals}, {\em J.
Algebra} {\bf 41} (1976), 439--454.  }



\bibitem{Eisenbudbook}{D. Eisenbud,  {\em Commutative Algebra with a
view toward Algebraic Geometry},
 Springer,  Berlin Heidelberg New
York, 1995.}



\bibitem{Elias}{J. Elias, On the canonical ideals of one-dimensional Cohen-Macaulay local rings, {\em Proc. Edinburgh Math. Soc.} {\bf 59} (2016), 77-90.}




 
\bibitem{FW93}{H. Flenner and W. Vogel, On multiplicities of local rings,
 {\em Manuscripta Math.} {\bf 78}  (1993), 85--97.}











\bibitem{red}{L. Ghezzi, S. Goto, J. Hong and W. V. Vasconcelos, 
 Sally modules and the reduction numbers of ideals, 
 {\em Nagoya Math. J.}
{\bf 226} (2017), 106--126.}


\bibitem{blue1} {L. Ghezzi, S. Goto, J. Hong  and W. V. Vasconcelos,  Invariants   of Cohen-Macaulay rings associated
to their canonical ideals,  {\em J. Algebra}  {\bf 589} (2017), 506--528.}


\bibitem{bideg} {L. Ghezzi, S. Goto, J. Hong, H. L. Hutson  and W. V. Vasconcelos, The bi-canonical degree of
   Cohen-Macaulay rings, {\em J. Algebra}, to appear.}
  





\bibitem{GMP11} {S. Goto, N. Matsuoka and T. T. Phuong, { Almost Gorenstein rings}, {\em  J. Algebra}
 {\bf 379} (2013), 355--381.}


\bibitem{GNO}{S. Goto, K. Nishida and K. Ozeki, The structure of Sally modules of rank one,
{ Math. Research Letters} {\bf 15} (2008), 881--892.
}

\bibitem{GTT15} {S. Goto, R. Takahashi and N. Taniguchi, { Almost Gorenstein rings -- towards a theory of higher dimension}, {\em J. Pure \& Applied Algebra}  {\bf 219} (2015), 2666--2712.}

 
 
\bibitem{Macaulay2}{D. Grayson and M. Stillman,
  {\em Macaulay 2}, a software system for research in algebraic
  geometry, 2006.
  Available at {http://www.math.uiuc.edu/Macaulay2/}.}

\bibitem{Herzog}{J. Herzog, Generators and relations of abelian semigroups and semigroup rings,
 {\em Manuscripta Math.} {\bf 3}  (1970), 175--193.}

\bibitem{Herzog16}{J. Herzog, T. Hibi and D. I. Stamade, The trace of the canonical module, {\em Israel
J. Math.} {\bf 233} (2019), 133--165.}

\bibitem{HK2}{J. Herzog and E. Kunz, {\em Der kanonische Modul eines Cohen--Macaulay Rings},
Lect. Notes in Math. {\bf 238}, Springer, Berlin--New York, 1971.}



\bibitem{HSV1}{J. Herzog, A. Simis and W. V. Vasconcelos, Koszul
homology and blowing-up rings, in  {\em Commutative Algebra},
Proceedings: Trento 1981
 (S. Greco and G. Valla, Eds.),
Lecture Notes in Pure and Applied Mathematics {\bf 84}, Marcel Dekker, New York,
1983, 79--169}.

\bibitem{HSV87}{J. Herzog, A. Simis and W. V. Vasconcelos,
On the canonical module of the  Rees algebras and the associated graded ring of an ideal,
{\em J. Algebra} {\bf 105} (1987), 285--302.}








 

 
 
 







\bibitem{Kapbook}{I. Kaplansky, {\em Commutative Rings}, The University of Chicago Press, Chicago, 1974.}

\bibitem{TK}{T. Kobayashi, Syzygies of Cohen-Macaulay modules over one dimensional  Cohen-Macaulay
local rings,}







\bibitem{MU}{S. Morey and B. Ulrich, Rees algebras of ideals with  low codimension, {\em Proc.  Amer. Math. Soc. } {\bf  124}  
(1996), 3653--3661.}


\bibitem{Nagata}{M. Nagata,  {\em Local Rings},  Interscience, New York, 1962.}


\bibitem{NW17}
{ S. Nasseh and S. Sather-Wagstaff, Geometric aspects of representation theory for DG algebras:
answering a question of Vasconcelos,
{\em J. Lond. Math. Soc.} {\bf  96} (2017),  271--292.}








 
 




\bibitem{dual}{A. Simis, B. Ulrich and W. V. Vasconcelos, Jacobian
dual fibrations, {\em Amer. J. Math.} {\bf 115} (1993), 47--75.}

 
 
 
\bibitem{Trung87}{
N. V. Trung, { Reduction exponent and degree bound for the defining equations of graded rings},
 {\em Proc. Amer. Math. Soc.} {\bf 101} (1987),  222--236.}
  
\bibitem{UlrichU96}{B. Ulrich, { Ideals having the expected reduction number}, {\em Amer.  J.  Math.}  {\bf 118} (1996), 117--138.}

\bibitem{V74}{W. V. Vasconcelos, {\em Divisor Theory in Module Categories}, North-Holland Publishing Co., Amsterdam, 1974, Notes in Mathematics {\bf 14}. 
}




\bibitem{Vas91}{W. V. Vasconcelos, { Computing the integral closure of an affine domain}, 
{\em Proc. Amer. Math. Soc.}  {\bf 113} (1991),  633--638.} 


 




\bibitem{compu}{ W. V. Vasconcelos, {\em Computational Methods in
Commutative Algebra and  Algebraic Geometry},
Springer, Heidelberg, 1998.}
 





\bibitem{VilaBook}{R. H. Villarreal, {\em Monomial Algebras}, 2nd Edition, 
 C.R.C. Press, New York, 2015.}
\end{thebibliography}
\end{document}